\documentclass[reqno, 12pt]{amsart}

\pdfoutput=1
\makeatletter
\let\origsection=\section \def\section{\@ifstar{\origsection*}{\mysection}} 
\def\mysection{\@startsection{section}{1}\z@{.7\linespacing\@plus\linespacing}{.5\linespacing}{\normalfont\scshape\centering\S}}
\makeatother        

\usepackage{amsmath,amssymb,amsthm}
\usepackage{mathrsfs}
\usepackage{mathtools}
\usepackage{mathabx}\changenotsign
\usepackage{dsfont}
\usepackage[foot]{amsaddr}
\usepackage{nicefrac}

\usepackage{graphicx}
\usepackage{xcolor}

\colorlet{myGreen}{green!50!black}
\colorlet{myLightgreen}{green}
\colorlet{myRed}{red!90!black}
\definecolor{myBlue}{rgb}{0.25, 0.0, 1.0}
\definecolor{myLightBlue}{rgb}{0.39, 0.58, 0.93}
\colorlet{myViolet}{myBlue!55!myRed}
\definecolor{myOrange}{rgb}{1.0, 0.66, 0.07}

\definecolor{CornflowerBlue}{rgb}{0.39, 0.58, 0.93}
\definecolor{DarkGoldenrod}{rgb}{0.72, 0.53, 0.04}
\definecolor{BritishRacingGreen}{rgb}{0.0, 0.26, 0.15}
\definecolor{DarkMagenta}{rgb}{0.55, 0.0, 0.55}
\definecolor{AO}{rgb}{0.0, 0.5, 0.0}
\definecolor{BostonUniversityRed}{rgb}{0.8, 0.0, 0.0}
\definecolor{myRed}{rgb}{0.8, 0.0, 0.0}
\definecolor{DarkMidnightBlue}{rgb}{0.0, 0.2, 0.4}
\definecolor{DarkTangerine}{rgb}{1.0, 0.66, 0.07}
\definecolor{AppleGreen}{rgb}{0.55, 0.71, 0.0}
\definecolor{BrightUbe}{rgb}{0.82, 0.62, 0.91}
\definecolor{Amethyst}{rgb}{0.6, 0.4, 0.8}
\definecolor{DarkGray}{rgb}{0.52, 0.52, 0.51}
\definecolor{Gray}{rgb}{0.66, 0.66, 0.66}
\definecolor{BananaYellow}{rgb}{1.0, 0.88, 0.21}
\definecolor{Amber}{rgb}{1.0, 0.75, 0.0}
\definecolor{LightGray}{rgb}{0.83, 0.83, 0.83}
\definecolor{PrincetonOrange}{rgb}{1.0, 0.56, 0.0}
\definecolor{DeepCarrotOrange}{rgb}{0.91, 0.41, 0.17}
\definecolor{MidnightBlue}{rgb}{0.1, 0.1, 0.44}
\definecolor{ElectricViolet}{rgb}{0.56, 0.0, 1.0}

\usepackage{verbatim}

\usepackage{enumitem}

\usepackage{hyperref}
\usepackage{cleveref}
\hypersetup{
    colorlinks,
    linkcolor={red!60!black},
    citecolor={green!60!black},
    urlcolor={blue!60!black},
}

\usepackage[open,openlevel=2,atend]{bookmark}

\input{_preamble/myBiblatex}

\usepackage[T1]{fontenc}
\usepackage{lmodern}
\usepackage[babel]{microtype}
\usepackage[english]{babel}

\usepackage{geometry}
\numberwithin{equation}{section}
\numberwithin{figure}{section}

\let\polishlcross=\l
\def\l{\ifmmode\ell\else\polishlcross\fi}

\def\paragraph#1{%
  \noindent\textbf{#1.}\enspace}

\makeatletter
\def\moverlay{\mathpalette\mov@rlay}
\def\mov@rlay#1#2{\leavevmode\vtop{   \baselineskip\z@skip \lineskiplimit-\maxdimen
   \ialign{\hfil$\m@th#1##$\hfil\cr#2\crcr}}}
\newcommand{\charfusion}[3][\mathord]{
    #1{\ifx#1\mathop\vphantom{#2}\fi
        \mathpalette\mov@rlay{#2\cr#3}
      }
    \ifx#1\mathop\expandafter\displaylimits\fi}
\makeatother

\usepackage{thmtools, thm-restate}

\theoremstyle{plain}
\newtheorem{thm}{Theorem}[section]
    \crefname{thm}{Theorem}{Theorems}
\newtheorem{theorem}[thm]{Theorem}
    \crefname{theorem}{Theorem}{Theorems}
\newtheorem{lemma}[thm]{Lemma}
    \crefname{lemma}{Lemma}{Lemmas}
\newtheorem{corollary}[thm]{Corollary}
    \crefname{corollary}{Corollary}{Corollaries}
\newtheorem{proposition}[thm]{Proposition}
    \crefname{proposition}{Proposition}{Propositions}

    \crefname{problem}{Problem}{Problems}

    \crefname{conjecture}{Conjecture}{Conjectures}
\newtheorem{observation}[thm]{Observation}
    \crefname{observation}{Observation}{Observations}
\newtheorem{question}[thm]{Question}
    \crefname{question}{Question}{Questions}
\newtheorem*{claim*}{Claim}
\newtheorem{claim}{Claim}[]
    \crefname{claim}{Claim}{Claims}
\newtheorem*{case*}{Case}

    \crefname{case}{Case}{Case}
\newtheorem{thm-intro}{Theorem}[]
    \crefname{thm-intro}{Theorem}{Theorems}
\newtheorem{conj-intro}[thm-intro]{Conjecture}
    \crefname{conj-intro}{Conjecture}{Conjectures}
\newtheorem{question-intro}[thm-intro]{Question}
    \crefname{question-intro}{Question}{Questions}

\theoremstyle{definition}
\newtheorem{definition}[thm]{Definition}
    \crefname{definition}{Definition}{Definitions}

    \crefname{remark}{Remark}{Remarks}

    \crefname{remarks}{Remarks}{Remarks}

    \crefname{situation}{Situation}{Situations}

    \crefname{construction}{Construction}{Constructions}

    \crefname{construction}{Example}{Examples}
\newtheorem*{example*}{Example}

\newenvironment{proofofclaim}[1][Proof.]{%
    \begin{proof}[{#1}]%
        }{%
    \end{proof}}

\DeclareFontFamily{U}  {MnSymbolC}{}
\DeclareSymbolFont{MnSyC}         {U}  {MnSymbolC}{m}{n}
\DeclareFontShape{U}{MnSymbolC}{m}{n}{
    <-6>  MnSymbolC5
   <6-7>  MnSymbolC6
   <7-8>  MnSymbolC7
   <8-9>  MnSymbolC8
   <9-10> MnSymbolC9
  <10-12> MnSymbolC10
  <12->   MnSymbolC12}{}
\DeclareMathSymbol{\powerset}{\mathord}{MnSyC}{180}

\let\setminus=\smallsetminus

\newcommand*{\abs}[1]{\ensuremath{{\lvert {#1} \rvert}}}

\newcommand*{\restricted}{\ensuremath{{\upharpoonright}}}

\hypersetup{
    pdftitle={Odd-Minors I},
    pdfauthor={J. Pascal Gollin and Sebastian Wiederrecht}
}

\newcommand*{\minor}{\leq_{\mathsf{minor}}}
\newcommand*{\odd}{\leq_{\mathsf{odd}}}

\begin{document}

\author[J.~P.~Gollin]{J.~Pascal Gollin$^1$}
\author[S.~Wiederrecht]{Sebastian Wiederrecht$^2$}
\address{$^1$FAMNIT, University of Primorska, Koper, Slovenia}
\address{$^2$Discrete Mathematics Group, Institute for Basic Science (IBS), Daejeon, South Korea}
\email{\tt pascal.gollin@famnit.upr.si}
\email{\tt sebastian.wiederrecht@gmail.com}

\thanks{The first author was supported by the Institute for Basic Science (IBS-R029-Y3) and in part by the Slovenian Research and Innovation Agency (research project N1-0370).
The second author was supported by the Institute for Basic Science (IBS-R029-C1).}

\newcommand{\oddminor}{odd-minor}
\newcommand{\Oddminor}{Odd-minor}
\newcommand{\OddMinor}{Odd-Minor}
\newcommand{\oddminors}{odd-minors}
\newcommand{\Oddminors}{Odd-minors}
\newcommand{\OddMinors}{Odd-Minors}

\title[Graphs excluding grids with small parity breaks as odd-minors]{Structure and algorithms for graphs excluding grids with small parity breaks as odd-minors}

\date{\today}

\keywords{odd-minors, parameterized complexity, graph parameter, maximum cut, maximum independent set}

\subjclass[2020]{05C83, 05C85, 05C75, 68R10, 68Q27}

\begin{abstract}
    We investigate a structural generalisation of treewidth we call \textit{$\mathcal{A}$-blind-treewidth} where~$\mathcal{A}$ denotes an \textit{annotated graph class}. 
    This width parameter is defined by evaluating only the size of those bags~$B$ of tree-decompositions for a graph~$G$ where~${(G,B) \notin \mathcal{A}}$.
    For the two cases where~$\mathcal{A}$ is (i) the class~$\mathcal{B}$ of all pairs~${(G,X)}$ such that no odd cycle in~$G$ contains more than one vertex of~${X \subseteq V(G)}$ and (ii) the class~$\mathcal{B}$ together with the class~$\mathcal{P}$ of all pairs~${(G,X)}$ such that the ``torso'' of~$X$ in~$G$ is planar. 
    For both classes, $\mathcal{B}$ and~${\mathcal{B} \cup \mathcal{P}}$, we obtain analogues of the Grid Theorem by Robertson and Seymour and FPT-algorithms that either compute decompositions of small width or correctly determine that the width of a given graph is large. 
    Moreover, we present FPT-algorithms for \textsc{Maximum Independent Set} on graphs of bounded $\mathcal{B}$-blind-treewidth and \textsc{Maximum Cut} on graphs of bounded ${(\mathcal{B}\cup\mathcal{P})}$-blind-treewidth. 
\end{abstract}

\maketitle

\section{Introduction}
\label{sec:intro}

In 1981 Gr\"otschel and Pulleyblank~\cite{GrotschelPulleyblank1981} showed that the class of so called \emph{weakly bipartite graphs} admits structural properties that allow for a polynomial time algorithm to solve the \textsc{Maximum Cut} (MC) problem. 
Later, Guenin~\cite{Guenin2001} resolved a conjecture of Seymour by proving that weakly bipartite graphs are precisely those that exclude~$K_5$ as an \textit{\oddminor}\footnote{This result is stated for the minor relation of signed graphs which is an equivalent formulation of the \oddminor\ relation as discussed in this paper.}. 
Roughly speaking, the \oddminor\ relation is a restriction of the minor relation of graphs that preserves parities of cycles. 
This means, in particular, that both planar graphs and bipartite graphs are closed under \oddminors\footnote{We postpone the formal definition of \oddminors\ to \Cref{sec:bipartite}.}. 

The particular properties of \oddminors, specifically their ``blindness'' towards graphs that are bipartite up to the deletion of a bounded number of vertices made them the subject of different directions of research. 
Notable here is the so called \textit{Odd Hadwiger's Conjecture} which asks whether every graph with chromatic number~$t$ contains~$K_t$ as an \oddminor~\cite{GeelenGRSV2009}. 

Demaine, Hajiaghayi, and Kawarabayashi~\cite{DemaineHK2010} approached the study of \oddminors\ with the tools of the Graph Minors Theory of Robertson and Seymour and showed that for every graph~$H$, every graph~$G$ that excludes~$H$ as an \oddminor\ can be obtained via clique sums from $H$-minor-free graphs and graphs which are bipartite up to the deletion of a small set of vertices. 
This theorem allowed Demaine et al.\@ to design a generic approximation algorithm for a wide collection of problems on graphs that exclude some fixed~$H$ as an \oddminor. 
Among these is, in particular, a PTAS for the \textsc{Maximum Independent Set} (MIS) problem.
The algorithms of Demaine et al.\@ were later improved to FPT-algorithms by Tazari~\cite{Tazari2012}. 

While the structural result of Demaine et al.\@ yields a powerful tool for the design of approximation algorithms on graphs excluding an \oddminor, exact algorithmic results on such graphs have been rather elusive. 
Moreover, a structural description of graphs excluding~$K_5$ as an \oddminor\ in the style of Wagner and Robertson and Seymour is yet to be discovered. 
Motivated by this gap we start the investigation of a more refined structural theory of graphs excluding an \oddminor. 
In this paper we make a first step by determining, asymptotically, the graphs whose exclusion as \oddminors\ yields families of graphs which are tree-decomposable into ``trivial'' pieces for an appropriate choice for the notion of ``trivial''. 
We show that this structural description yields FPT-approximation algorithms for the related width-parameters and apply the resulting decompositions for the design of FPT-algorithms for the MIS and MC problem. 

\subsection*{%
\texorpdfstring{$\mathcal{A}$-blind treewidth}%
{A-blind treewidth}}

To properly describe the ``trivial'' pieces as substructures within a graph we make use of annotated graphs as follows.
An \emph{annotated graph} is a tuple~${(G,X)}$ where~$G$ is a graph and~${X \subseteq V(G)}$. 
An \emph{annotated graph class} is a class~$\mathcal{A}$ of annotated graphs.

This notation allows us the describe structural properties of single bags of a given tree-decomposition with respect to the surrounding graph.

A \emph{tree-decomposition} of a graph~$G$ is a tuple~$(T,\beta)$ of a tree~$T$ and a map~$\beta$ that assigns to the nodes of~$T$ subsets of~${V(G)}$, called the \emph{bags}, such that~${\bigcup_{t \in V(T)} \beta (t) = V(G)}$, for every~${e \in E(G)}$ there is a~${t \in V(T)}$ with~${e \subseteq \beta(t)}$, and for every~${v \in V(G)}$, the set ${\{t \in V(T) \mid v \in \beta(t) \}}$ is connected. 
The \emph{width} of~${(T,\beta)}$ is the value~${\max_{t \in V(T)} \abs{\beta(t)} - 1}$. 
The \emph{treewidth} of~$G$, denoted by~${\mathsf{tw}(G)}$, is the minimum width over all tree-decompositions of~$G$.

For an annotated graph class~$\mathcal{A}$ we extend the notion of treewidth as follows. 
The \emph{$\mathcal{A}$-blind-width} of a tree-decomposition $(T,\beta)$ for a graph~$G$ is the largest size of a bag~$\beta(t)$, such that~${(G,\beta(t)) \notin \mathcal{A}}$, or~$0$ if ${(G,\beta(t)) \in \mathcal{A}}$ for every~${t \in V(T)}$. 
The \emph{$\mathcal{A}$-blind-treewidth} of a graph~$G$, denoted by ${\mathcal{A}\text{-}\mathsf{blind}\text{-}\mathsf{tw}(G)}$, is the smallest integer~$k$ such that~$G$ has a tree-decomposition of~$\mathcal{A}$-blind-width~$k$.

In this paper, we are mostly concerned with the following two natural classes. 
By~$\mathcal{B}$ we denote the class of annotated graph~${(G,X)}$ where any odd cycle in~$G$ contains at most one vertex of~$X$. 
If~${(G,X) \in \mathcal{B}}$ we say that~$X$ \emph{globally bipartite}.
Note that if~$X$ is a globally bipartite set in~$G$ where~${G[X]}$ is $2$-connected, then for each odd cycle~$C$ of~$G$ there is a vertex whose deletion separates~$X$ from~$C$. 
This separation property (manifested in Observation~\ref{obs:blind-tw-blocks}) is what facilitates dynamic programming on tree-decompositions of small $\mathcal{B}$-blind-width. 

The \emph{torso} of a set~${X \subseteq V(G)}$ in a graph~$G$ is the graph~$G_X$ obtained from~${G[X]}$ by turning every set~${S \subseteq X}$ such that there exists a component~$H$ of~${G-X}$ with~${N_G(V(H)) = S}$ into a clique. 
By~$\mathcal{P}$ we denote the class of annotated graph $(G,X)$ such that the torso of $X$ in $G$ is a planar.

While it has never been explicitly stated as a graph parameter, the seminal result of Robertson and Seymour on \textit{single-crossing-minor-free graphs}\footnote{A graph is a \emph{single-crossing-minor} if it is a minor of some graph that can be drawn in the plane with a single crossing.} \cite{RobertsonS1993} can be seen as the first for $\mathcal{A}$-blind treewidth as they provide an asymptotic description of $\mathcal{P}$-blind-treewidth which we discuss later.

\subsection*{Our results}
Our algorithmic main results are making tree-decomposition of small $\mathcal{B}$-blind- and $(\mathcal{B}\cup\mathcal{P})$-blind-width accessible and apply them to solve the two problems MIS and MC, respectively. 

\begin{theorem}
    \label{thm:maxindependentset}
    There exists a computable function~$f$ and an algorithm that, given a positive integer~$k$ and a graph~$G$ as inputs, in time~${\mathcal{O}(f(k) \cdot \abs{V(G)}^3)}$, either finds a maximum independent set of~$G$, 
    or concludes correctly that~${\mathcal{B}\text{-}\mathsf{blind}\text{-}\mathsf{tw}(G) > k}$. 
\end{theorem}

\begin{theorem}
    \label{thm:maxcut}
    There exists a computable function~$f$ and an algorithm that, given a positive integer~$k$ and a graph~$G$ as inputs, in time~${\mathcal{O}(f(k) \cdot \abs{V(G)}^{\nicefrac{9}{2}})}$, either finds a maximum cut of~$G$, 
    or concludes correctly that~${(\mathcal{B}\cup\mathcal{P})\text{-}\mathsf{blind}\text{-}\mathsf{tw}(G) > k}$. 
\end{theorem}

Notice that every tree-decomposition of small $\mathcal{B}$-blind-width is also, in particular, of small $(\mathcal{B}\cup\mathcal{P})$-blind-width. 
Thus, \Cref{thm:maxcut} finds application also for all graphs where \Cref{thm:maxindependentset} can be applied. 

Neither of the two theorems above requires an $\mathcal{A}$-blind tree-decomposition of small width to be given as part of the input, where~$\mathcal{A}$ denotes the respective class  $\mathcal{B}$ or~${\mathcal{B} \cup \mathcal{P}}$. 
This is, because of our proofs are constructive and yield, among other things, algorithms that allow to find a decomposition of small width if one exists. 

\begin{theorem}\label{thm:computedecompositions}
    For each~${\mathcal{A} \in \{ \mathcal{B}, \mathcal{B} \cup \mathcal{P} \}}$ there exist computable functions~${f_{\mathcal{A}},g_{\mathcal{A}}}$ and an algorithm that, 
    given a positive integer~$k$ and a graph~$G$ and inputs, 
    in time $\mathcal{O}(g_{\mathcal{A}}(k)\abs{V(G)}^3)$, 
    either finds a tree decomposition of $\mathcal{A}$-blind-width at most $f(k)$, 
    or concludes correctly that~${\mathcal{A}\text{-}\mathsf{blind}\text{-}\mathsf{tw}(G) > k}$. 
\end{theorem}

As announced in the beginning, we determine, asymptotically, all odd-minor-closed graph classes where $\mathcal{A}$-blind-treewidth is bounded for $\mathcal{A}\in\{\mathcal{B},\mathcal{B}\cup\mathcal{P} \}$.
We achieve this by proving odd-minor analogues of the celebrated Grid Theorem of Robertson and Seymour \cite{RobertsonS1986:GraphMinorsV} for both parameters.
Before we proceed with our structural results and how those are achieved, let us quickly give some context to the position our variation of treewidth takes within the landscape of parameterized complexity and width parameters.

\subsection*{Hybridisation in parameterized complexity}
The technique of relaxing an established parameter such as treewidth by, essentially, ignoring parts that belong to a fixed graph class~$\mathcal{C}$ as long as these parts interact well with the rest of the graph has become known as \textit{hybridisation} in parameterized complexity. 
One may regard results such as Kaminski's algorithm for \textsc{Maximum Cut} on single-crossing-minor-free graph classes \cite{Kaminski2012} as an example. 
A broad generalisation of this approach in minor-closed classes is the recent dichotomy result for counting perfect matchings in minor-closed graph classes by Thilikos and Wiederrecht~\cite{ThilikosW2022}. 

Another recent article by Eiben, Ganian, Hamm, and Kwon~\cite{EibenGHK2021} introduced the notion of \textit{$\mathcal{H}$-treewidth} where $\mathcal{H}$ is a (hereditary) graph class. 
Here, the authors ask for a tree-decomposition in which the bags of all internal nodes are small, while the leaf bags, after deleting their adhesion set\footnote{The adhesion set of a bag $B$ in a tree-decomposition is the union of the intersections of $B$ with the neighbouring bags.}, must belong to~$\mathcal{H}$ if they are large. 
This concept was further studied in~\cite{JansenKW2021}.
Note that $\mathcal{H}$-treewidth is a measure of the complexity of a modulator whose deletion results in all components being in~$\mathcal{H}$.
In contrast, for $\mathcal{A}$-blind-treewidth, also internal nodes are allowed to be of arbitrary size as long as they are in~$\mathcal{A}$.

In the above mentioned result on counting perfect matchings, the graph classes under consideration were graphs which are of bounded Euler-genus after the deletion of a small apex set. 
This \textit{modulation} towards a fixed graph class is also the main engine behind the ideas of Eiben et al.\@ and Jansen et al. 
A different approach could be taken by changing the evaluation of the bags of a tree-decomposition completely. 
This is done, for example, in the definition of \textit{$\alpha$-treewidth} where, instead of the size of a bag, one evaluates the size of a maximum independent set of its induced subgraph. 
This notion was introduced as another way to parameterize the \textsc{Maximum Independent Set} problem by means of tree-decompositions~\cite{Yolov2018,DallardMS2021}. 

\subsection*{The parametric perspective and our structural results}
Tree-decompositions of graphs allow for the almost automated design of dynamic programming algorithms and hence yield efficient algorithms when parameterised by the treewidth of the graph~\cite{Courcelle1990,Bodlaender2007}. 

The Grid Theorem of Robertson and Seymour~\cite{RobertsonS1986:GraphMinorsV} asserts the existence of a family~${ \mathscr{G} = \langle \mathscr{G}_i \rangle_{i \in \mathbb{Z}^{+}}}$ where~$\mathscr{G}_k$ is the ${(k \times k)}$-grid, such that every graph with treewidth at least~${f(k)}$ contains~$\mathscr{G}_k$ as a minor. 
Moreover, the containment of~$\mathscr{G}_k$ as a minor certifies that the treewidth of a graph is at least~$k$. 
Let us denote by~${\mathsf{big}_\mathscr{G}(G)}$ the largest integer~$k$ such that~$G$ contains~$\mathscr{G}_k$ as a minor. 

We say that two graph parameters~$\mathsf{p}$ and~$\mathsf{q}$ are \emph{asymptotically equivalent}, and write~${\mathsf{p} \sim \mathsf{q}}$, if there is a function~$f$ such that~${\mathsf{p}(G) \leq f(\mathsf{q}(G))}$ and~${\mathsf{q}(G) \leq f(\mathsf{p}(G))}$ for all graphs~$G$. 

In this notation, the Grid Theorem of Robertson and Seymour reduces to
\begin{align*}
    \mathsf{big}_\mathscr{G}(G) \sim  \mathsf{tw}.
\end{align*}
This means that, asymptotically, the treewidth of a graph, and thus the algorithmic power of treewidth, is obstructed by a single family of graphs; the grids~$\mathscr{G}$.

\begin{figure}[ht]
    \begin{center}
        \includegraphics{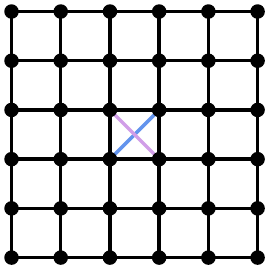}
    \end{center}
    \caption{The single-crossing grid~$\mathscr{U}_3$ of order~$3$. The single-crossing grid~$\mathscr{U}_k$ is obtained from the $(2k \times 2k)$-grid by adding a crossing to the centremost face. }
    \label{figure:singlecrossinggrid}
\end{figure}

As a first extension of the Grid Theorem and a fragment of the general Graph Minors Structure Theorem~\cite{RobertsonS2003:GraphMinorsXVI,KawarabayashiTW2020}, Robertson and Seymour~\cite{RobertsonS1993} described the minor-closed graph classes that allow for tree-decompositions into small pieces and pieces that are planar. 
They identified another family~${\mathscr{U} = \langle \mathscr{U}_i \rangle_{i \in \mathbb{Z}^{+}}}$, see \Cref{figure:singlecrossinggrid} for an illustration, such that excluding some member of~$\mathscr{U}$ allows for a decomposition of small $\mathcal{P}$-blind-width while the $\mathcal{P}$-blind-width of~$\mathscr{U}_k$ is at least~${2k+1}$ (see \Cref{lem:P-tw-single-crossing-grid}). 
Similar to the relation of treewidth with $\mathsf{big}_\mathscr{G}(G)$, this yields an asymptotic equivalence between $\mathcal{P}\text{-}\mathsf{blind}\text{-}\mathsf{tw}$ and the maximisation parameter defined through the minor-containment of a member of~$\mathscr{U}$. 
Similar to graphs of bounded treewidth, the structure of these so called \textit{single-crossing-minor-free graphs} allows for powerful algorithmic applications~\cite{Kaminski2012,DemaineHNRT2004}.

For $\mathcal{P}$-blind-treewidth, the minor-relation is the natural containment relation as classes of bounded treewidth and the class of planar graphs are closed under minors. 
However, once we introduce bipartite graphs as one of our target classes we need to be more careful. 
For $\mathcal{B}$-blind-treewidth and ${(\mathcal{B}\cup \mathcal{P})}$-blind-treewidth, \oddminors\ yield the appropriate containment relation, as preserving the parity of cycles ensures that the class of bipartite graphs is closed under \oddminors. 

\begin{figure}[ht]
    \begin{center}
        \includegraphics[scale=1]{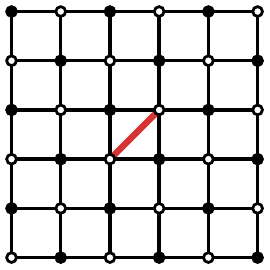}
    \end{center}
    \caption{The singly parity-breaking~$\mathscr{S}_3$. The red edge is the sole edge joining two vertices from the same colour class of a proper $2$-colouring of the underlying grid, hence, every odd cycle has to contain that edge. }
    \label{figure:singleoddintro}
\end{figure}

The \emph{singly parity-breaking grid}~$\mathscr{S}_k$ of \emph{order~$k$} is the graph obtained from the ${(2k\times 2k)}$-grid by adding 
a single parity-breaking edge in the centremost face. 
We denote the family of singly parity-breaking grids by~${\mathscr{S} \coloneqq \langle \mathscr{S}_i \rangle_{i \in \mathbb{Z}^{+}}}$. 
See \Cref{figure:singleoddintro} for an illustration. 
Moreover, we define the parameter~${\mathsf{big}_\mathscr{S}(G)}$ of a graph~$G$ to be the largest integer~$k$ such that~$G$ contains~$\mathscr{S}_k$ as an \oddminor. 

This allows us to state our first structural result as a asymptotic characterisation of those graphs that can be decomposed into bipartite graphs and graphs of small treewidth. 

\begin{theorem}
    \label{thm:singleoddstructure}
    It holds that ${\mathcal{B}\text{-}\mathsf{blind}\text{-}\mathsf{tw} \sim \mathsf{big}_{\mathscr{S}}}$. 
    In particular, there exist a computable function ${f \colon \mathbb{N} \to \mathbb{N}}$
    such that for every positive integer~$k$, every graph~$G$ either contains contains the singly parity-breaking grid of order~$k$ as an \oddminor, or ${\mathcal{B}\text{-}\mathsf{blind}\text{-}\mathsf{tw}(G) \leq f(k)}$.  
\end{theorem}

\begin{figure}
    \begin{center}
        \includegraphics[scale=1]{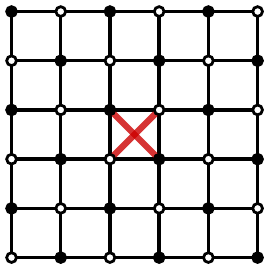}\qquad
        \includegraphics[scale=1]{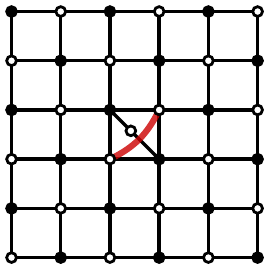}\qquad
        \includegraphics[scale=1]{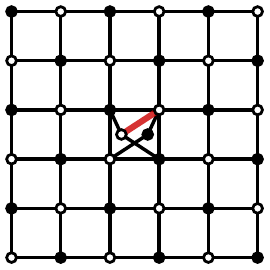}
    \end{center}
    \caption{The three singly parity-crossing grids~$\mathscr{C}^1_3$, $\mathscr{C}^2_3$, and~$\mathscr{C}^3_3$ (from left to right). 
    The red edges are those joining two vertices of the same colour class in proper $2$-colouring of the underlying grid.}
    \label{figure:singleoddcross}
\end{figure}

For our second structural result we need three different families~${\mathscr{C}^i = \langle \mathscr{C}^i_j \rangle_{j \in \mathbb{Z}^{+}}}$ for~${i \in [3]}$ which are obtained from the ${(2k\times 2k)}$-grid by adding a cross with a small parity break in the centremost, see \Cref{figure:singleoddcross} for an illustration. 
We call~${\mathfrak{C} \coloneqq\{ \mathscr{C}^i \mid i\in [3] \}}$ the family of \emph{singly parity-crossing grids} and say that a graph~$G$ contains a singly parity-crossing grid \emph{of order~$k$} as an \oddminor\ if there exists~${i \in [3]}$ such that~$\mathscr{C}^i_k$ is an \oddminor\ of~$G$. 
Moreover, we define the parameter~${\mathsf{big}_{\mathfrak{C}}(G)}$ of a graph~$G$ to be the largest integer~$k$ such that~$G$ contains~$\mathscr{C}^i_k$ as an \oddminor\ for some~${i \in [3]}$. 

This allows us to state our second structural result as a asymptotic characterisation of those graphs that can be decomposed into bipartite graphs, planar graphs, and graphs of small treewidth.

\begin{theorem}
    \label{thm:singleoddcrossstructure}
    It holds that ${(\mathcal{B} \cup \mathcal{P})\text{-}\mathsf{blind}\text{-}\mathsf{tw} \sim \mathsf{big}_{\mathscr{S}}}$. 
    In particular, there exist a computable function~${f \colon \mathbb{N} \to \mathbb{N}}$
    such that for every positive integer~$k$, every graph~$G$ either contains contains a singly parity-crossing grid of order~$k$ as an \oddminor, or ${(\mathcal{B} \cup \mathcal{P})\text{-}\mathsf{blind}\text{-}\mathsf{tw}(G) \leq f(k)}$. 
\end{theorem}

The proof of \Cref{thm:singleoddstructure} and \Cref{thm:singleoddcrossstructure} are constructive in the sense that they provide algorithms which, for some function $g$, find in time $g(k)\abs{V(G)}^{\mathcal{O}(1)}$ an \oddminor\ from the desired family, or conclude correctly that no such \oddminor\ exists.
The specific constructions yield therefore the proof of \Cref{thm:computedecompositions}.

\subsection*{Organisation}

In \Cref{sec:prelims} we introduce most definitions and previous results necessary for our proofs. 
Here, we also discuss Robertson and Seymour's theory of single-crossing-minor-free graphs in terms of $\mathcal{P}$-blind-treewidth. 
Towards our structural results we first briefly discuss some properties of \oddminor{s} of bipartite graphs in \Cref{sec:bipartite}. 
We then extract an important lemma on finding parity-breaking ears of non-bipartite $2$-connected graphs in \Cref{sec:ears}. 
Finally, we prove \cref{thm:singleoddstructure} in \Cref{sec:singleodd} and \cref{thm:singleoddcrossstructure} in \Cref{sec:oddcrossing}. 
For the algorithmic applications of our structural results, the proof of \cref{thm:maxindependentset} can be found in \Cref{sec:maxindependentset} while the proof for \cref{thm:maxcut} is contained in \Cref{sec:maxcut}. 

\section{Preliminaries}
\label{sec:prelims}

Throughout this paper, a graph~$G$ will be finite and simple with vertex set~${V(G)}$ and edge set~${E(G)}$. 
For standard graph theoretic notions that we are not defining, we refer the reader to~\cite{Diestel:GraphTheory5}. 

We denote by~$\mathbb{N}$ and~$\mathbb{Z}^+$ the set of non-negative and positive integers, respectively
For~${m, n \in \mathbb{N}}$, we write~${[n]}$ for the set~${\{ i \in \mathbb{N} \mid 1 \leq i \leq n \}}$, and we write~${[m,n]}$ for the set~${\{ i \in \mathbb{N} \mid m \leq i \leq n \}}$. 

\subsection{\Oddminors}
Let~$G$ and~$H$ be graphs. 
An \emph{$H$-expansion in~$G$} is a map~$\eta$ from ${V(H) \cup E(H)}$ to the set of non-empty subgraphs of~$G$ such that
\begin{itemize}
    \item $\eta(v)$ is a tree in~$G$ for each~${v \in V(H)}$, and~$\eta(v)$ is vertex-disjoint from~$\eta(w)$ for distinct~${v,w \in V(H)}$; 
    \item $\eta(vw)$ is an edge of~$G$ for each~${vw \in E(H)}$ that contains a vertex from both~$\eta(v)$ and~$\eta(w)$. 
\end{itemize}
We denote by~$H_\eta$ the subgraph of~$G$ that is the union~${\bigcup_{v \in V(H)} \eta(v)}$ together with the edges in~$\eta(E(H))$
and say~$H_\eta$ is an \emph{inflated copy of~$H$}.
We call~$\eta(v)$ the \emph{branch set} of~$v$.

We say~$H$ is a \emph{minor} of~$G$, and write~${H \minor G}$, if there is an $H$-expansion in~$G$.

We call an $H$-expansion in~$G$ \emph{odd} if there is a $2$-colouring~$c$ of~$H_\eta$ such that ${c \restricted \eta(v)}$ is a proper $2$-colouring for each~${v \in V(H)}$ and~$\eta(e)$ is monochromatic for each~${e \in E(H)}$, and we say~$c$ \emph{witnesses} that~$\eta$ is odd. 

We say~$H$ is an \emph{\oddminor} of~$G$, and write~${H \odd G}$, if there is an odd $H$-expansion in~$G$. 
Clearly, $\odd$ is a quasi-order.

We say that a graph~$H'$ is a \emph{subdivision} of another graph~$H$ if~$H'$ can be obtained from~$H$ be replacing the edges of~$H$ with paths whose set of internal vertices is disjoint from the rest of the graph. 
We say that an edge is \emph{subdivided $n$-times} if it was replaced by a path of length~${n+1},$ that is, if it was replaced by a path with~$n$ internal vertices.
In case~$H$ is bipartite and~$H'$ is a subdivision of~$H$ where each edge is subdivided an even number of times we say that~$H'$ is a \emph{bipartite subdivision} of~$H$. 

\subsection{Connectivity and tree-decompositions}

Let~$G$ be a graph and~${k \geq 1}$ be an integer. 
We say that~$G$ is \emph{$k$-connected} if~${\abs{V(G)} \geq k+1}$ and~${G-S}$ is connected for every set~${S \subseteq V(G)}$ of size at most~${k-1}$. 
A vertex~${v \in V(G)}$ is called a \emph{cutvertex of~$G$} if~${G-v}$ has more components than~$G$.
We call a subgraph~$H$ of a graph~$G$ a \emph{block of~$G$} if $H$ is a maximal subgraph~$G$ that does not contain a cutvertex of~$H$. 
Hence, a block of~$G$ is either isomorphic to~$K_2$ or $2$-connected. 

A graph~$G$ is said to be \emph{quasi-$4$-connected} if it is $3$-connected and for all 
sets~${S \subseteq V(G)}$ of size~$3$, one of the components of~${G-S}$ has size at most~$1$.

\begin{definition}[Tree-decomposition]\label{def_treewidth} 
    Let~$G$ be a graph. 
    A \emph{tree-decomposition} of~$G$ is a tuple~${\mathcal{T} = (T,\beta)}$ where~$T$ is a tree and~${\beta \colon V(T) \to 2^{V(G)}}$ is a function, whose images are called the \emph{bags} of~$\mathcal{T}$, such that 
    \smallskip
    \begin{enumerate}
        \item ${\bigcup_{t \in V(T)} \beta (t) = V(G)}$,  
        \item for every~${e \in E(G)}$ there exists~${t \in V(T)}$ with~${e \subseteq \beta(t)}$, and 
        \item for every~${v \in V(G)}$ the set~${\{t \in V(T) \mid v \in \beta(t) \}}$ induces a subtree of~$T$. 
    \end{enumerate} 
    \smallskip
    We refer to the vertices of~$T$ as the \emph{nodes} of the tree-decomposition~$\mathcal{T}$. 
    
    A tuple~${\mathcal{T} = (T,r,\beta)}$ is a \emph{rooted tree-decomposition} for a graph~$G$ if~${(T,\beta)}$ is a tree-decomposition for~$G$ and~${r \in V(T)}$. 
    We call~$r$ the \emph{root} of~$\mathcal{T}$. 
\end{definition}

Let~$G$ be a graph and~${(T,\beta)}$ be a tree-decomposition of~$G$.

For each~${t \in V(T)}$, we define the \emph{adhesion sets of~$t$} as the sets in ${\{ \beta(t) \cap \beta(d) \mid d \in N_T(t) \}}$ and the maximum size of them is called the \emph{adhesion of~$t$}. 
The \emph{adhesion} of~$\mathcal{T}$ is the maximum adhesion of a node of~$\mathcal{T}$. 

The \emph{torso} of~$\mathcal{T}$ on a node~$t$ is the graph~$G_t$ obtained by adding edges between every pair of vertices of~${\beta(t)}$ which belongs to a common adhesion set of~$t$.
Notice that the torso of $\mathcal{T}$ at $t$ contains the torso of $\beta(t)$ in $G$ as a subgraph.

The \emph{width} of~${(T,\beta)}$ is the value~${\max_{t \in V(T)} \abs{\beta(t)} - 1}$. 
The \emph{treewidth} of~$G$, denoted by~${\mathsf{tw}(G)}$, is the minimum width over all tree-decompositions of~$G$. 

\begin{proposition}[\cite{Bodlaender1996}]
    \label{prop:treewidth}
    There exists a computable function~$g_{\ref{prop:treewidth}}$ and an algorithm that, given a positive integer~$k$ and a graph~$G$ as inputs, in time~${\mathcal{O}(g_{\ref{prop:treewidth}}(k) \abs{V(G)})}$, either finds a tree-decomposition of width~$k$ for~$G$, or concludes correctly that~${\mathsf{tw}(G) > k}$. 
\end{proposition}

The following is a reformulation of Tutte's block-cut-tree theorem in the language of tree-decompositions. 

\begin{proposition}[follows from~\cite{Schmidt2013}]\label{prop:cutvertextree}
    Let~$G$ be a connected graph.
    Then there exists a tree-decomposition~${(T,\beta)}$ of~$G$ of adhesion one such that, for every~${t \in V(T)}$, the torso of~$G$ on the node~$t$ is a block of~$G$.
    Moreover, this decomposition can be found in time~${\mathcal{O}(\abs{V(G)} + \abs{E(G)})}$. 
\end{proposition}

Moreover, we will need the following extension of \cref{prop:cutvertextree} by Grohe. 

\begin{proposition}[\cite{Grohe2016}]
    \label{prop:4components}
    Every graph~$G$ has a tree-decomposition~${(T,\beta)}$ of adhesion at most~$3$ such that for all~${t \in V(T)}$ the torso of~$G$ on the node~$t$ is a minor of~$G$ that is either quasi-$4$-connected or isomorphic to a complete graph on at most four vertices. 
    Moreover, this decomposition can be found in time~${\mathcal{O}(\abs{V(G)}^3)}$. 
\end{proposition}

Let~$G$ be a graph and let~${(T,\beta)}$ be the tree-decomposition of~$G$ provided by \cref{prop:4components}, we call the torsos of~$G$ at the vertices of~$T$ the \emph{quasi-$4$-components} of~$G$. 
Grohe further shows that, up to isomorphism, these are independent of the choice of tree-decomposition, therefore the name quasi-$4$-components of~$G$ is justified. 

\subsection{Parametric graphs and their families}

Finally, we need a formal way to deal with graph parameters and their (universal) obstructions. 
Let us denote by~$\mathcal{G}_{\text{all}}$ the class of all finite graphs. 

\begin{definition}[Graph parameter]
    A \emph{graph parameter} is a function mapping finite graphs to non-negative integers that is invariant under graph isomorphisms. 
    
    Let~$\preceq$ denote a quasi-order on~$\mathcal{G}_{\text{all}}$, we say that such an order is a \emph{graph containment relation}. 
    
    A graph parameter~${\mathsf{p} \colon \mathcal{G}_{\text{all}} \to \mathbb{N}}$ is \emph{$\preceq$-monotone} if for every graph~$G$ and every~$H$ with~${H \preceq G}$ we have~${\mathsf{p}(H) \leq \mathsf{p}(G)}$. 
    
    Let~$\mathsf{p}$ and~$\mathsf{q}$ be two graph parameters. 
    We say that~$\mathsf{p}$ and~$\mathsf{q}$ are \emph{asymptotically equivalent}, denoted by~${\mathsf{p} \sim \mathsf{q}}$, if there is a function~${f \colon \mathbb{N} \to \mathbb{N}}$ such that, for every graph~$G$ it holds that~${\mathsf{p}(G) \leq f(\mathsf{p}(G))}$ and~${\mathsf{q}(G) \leq f(\mathsf{p}(G))}$. 
\end{definition}

The two graph containment relations we are concerned with in this paper are the minor relation, denoted by~$\minor$, and the \oddminor\ relation, denoted by~$\odd$. 

\begin{definition}[Parametric graph]
    A \emph{parametric graph} is a sequence~${\mathscr{H} = \langle \mathscr{H}_{i} \rangle_{i \in \mathbb{Z}^{+}}}$ of graphs indexed by positive integers. 
    We say that~$\mathscr{H}$ is \emph{$\preceq$-monotone} if for every~${i \in \mathbb{Z}^{+}}$ we have~${\mathscr{H}_i \preceq \mathscr{H}_{i+1}}$. 
\end{definition}

Let~${\mathscr{H}^{(1)} = \langle \mathscr{H}_{k}^{(1)} \rangle_{k \in \mathbb{Z}^{+}}}$ 
and~${\mathscr{H}^{(2)} = \langle \mathscr{H}_{k}^{(2)} \rangle_{k \in \mathbb{Z}^{+}}}$ be two parametric graphs. 
We write ${\mathscr{H}^{(1)} \precsim \mathscr{H}^{(2)}}$ if there is a function~${f \colon \mathbb{Z}^{+} \to \mathbb{Z}^{+}}$ such that~${\mathscr{H}_{k}^{(1)} \preceq \mathscr{H}_{f(k)}^{(2)}}$ for every~${k \in \mathbb{Z}^{+}}$. 
We say that~$\mathscr{H}^{(1)}$ and~$\mathscr{H}^{(2)}$ are \emph{asymptotically equivalent} 
if~${\mathscr{H}^{(1)} \precsim \mathscr{H}^{(2)}}$ 
and~${\mathscr{H}^{(2)} \precsim \mathscr{H}^{(1)}}$ and we denote this by~${\mathscr{H}^{(1)} \approx \mathscr{H}^{(2)}}$. 

A \emph{parametric family} is a finite collection~${\mathfrak{H} = \{ \mathscr{H}^{(j)} \mid j \in [\ell] \}}$ of parametric graphs with~${\mathscr{H}^{(j)} = \langle \mathscr{H}^{(j)}_{i} \rangle_{i \in \mathbb{Z}^{+}}}$ for all~${j \in [\ell]}$, and~${\mathscr{H}^{(i)} \not\precsim \mathscr{H}^{(j)}}$ and~${\mathscr{H}^{(j)} \not\precsim \mathscr{H}^{(i)}}$ for all distinct~${i,j \in [\ell]}$. 
If each parametric graph in~$\mathfrak{H}$ is $\preceq$-monotone, we call~$\mathfrak{H}$ a \emph{$\preceq$-parametric family}. 
We define a graph parameter by setting~${\preceq\text{-}\mathsf{big}_\mathfrak{H}(G)}$ to be the largest integer~$k$ for which there is a ${j \in [\ell]}$. 
By slight abuse of notation, we drop the graph containment relation~$\preceq$ from the name of the parameter, if it is clear from the context. 
In particular, if in this paper we use the parameter without reference to a graph containment relation, we always mean the \oddminor\ relation~$\odd$.
For more on parametric families see \cite{PaulPT2023}.

\subsection{Grids and walls}
\label{subsec:walls}

The Grid Theorem asserts that every graph of large enough treewidth, with respect to some fixed integer $k$, contains every planar graph on~$k$ vertices as a minor. 
The aim of this subsection is to establish a similar theorem for planar and bipartite graphs in the case of the \oddminor\ relation. 

We start by discussing how one can obtain a grid as an \oddminor. 
To do this we use a slight reformulation of the Grid Theorem in terms of walls as these can be found as subgraphs rather than minors which allows for easier arguments later. 
\smallskip

An \emph{$(n\times m)$-grid} is the graph with vertex set~${[n] \times [m]}$ and edge set 
\[
    {\{ \{ (i,j), (i,j+1) \} \mid i \in [n], j \in [m-1] \} \cup \{ \{ (i,j), (i+1,j) \} \mid i \in [n-1], j \in [m] \}}.
\]
We denote by~$\mathscr{G}_k$ the ${(k \times k)}$-grid. 

The \emph{elementary $k$-wall}~$\mathscr{W}_k$ for~${k \geq 3},$ is obtained from the ${(k \times 2k)}$-grid~$G_{k,2k}$ by deleting every odd edge in every odd column and every even edge in every even column, and then deleting all degree-one vertices. 
The \emph{rows} of~$\mathscr{W}_k$ are the subgraphs of~$\mathscr{W}_k$ induced by the rows of~$G_{k,2k}$, 
while the \emph{$j$th column} of~$\mathscr{W}_k$ is the subgraph induced by the vertices of columns~${2j-1}$ and~${2j}$ of~$G_{k,2k}$. 
The \emph{perimeter} of~$\mathscr{W}_k$ is the union of the first column, the $k$th column, the first row, and the $k$th row of~$\mathscr{W}_k$. 
The \emph{corners} of~$\mathscr{W}_k$ are the endvertices of the first and $k$th rows (or equivalently, columns). 
A \emph{brick} of~$\mathscr{W}_k$ is a facial six-cycle of~$\mathscr{W}_k$. 
Given the elementary $2k$-wall~$\mathscr{W}_{2k}$, the \emph{central brick} is the brick containing the vertices~${(k,2k)}$, ${(k,2k+2)}$, ${(k+1,2k+2)}$, and~${(k+1,2k)}$. 

A \emph{$k$-wall}~$W$ is a graph isomorphic to a subdivision of~$\mathscr{W}_k$. 
The vertices of degree three in~$W$ are called the \emph{branch vertices}. 
In other words, $W$ is obtained from a graph~$W'$ isomorphic to~$\mathscr{W}_k$ by subdividing each edge of~$W'$ an arbitrary (possibly zero) number of times. 
We define rows, columns, corners, the perimeter, and bricks of $k$-walls as well as the central brick of a $2k$-wall as the objects corresponding to the respective object of the elementary $k$-wall~$\mathscr{W}_k$ (or the elementary $2k$-wall in the case of central bricks) by implicitly fixing an isomorphism to the subdivision of~$\mathscr{W}_k$. 
A \emph{wall} is a $k$-wall for some~$k$. 

An $h$-wall~$W'$ is a \emph{subwall} of some $k$-wall~$W$ where~${h \leq k}$ if every row (column) of~$W'$ is contained in a row (column) of~$W$ (where the implicit isomorphism between~$W'$ and the subdivision of~$W_h$ is the one induced by the implicit isomorphism between~$W$ and the subdivision of~$\mathscr{W}_k$). 

We are interested in bipartite $k$-walls which adhere to certain additional conditions.

\begin{definition}[Clean wall]
    Let~${k \geq 3}$ be an integer.
    A $k$-wall~$W$ is \emph{clean} if there exists a proper $2$-colouring~$c$ of~$W$ such that all branch vertices and corners are in the same colour class of~$c$. 
\end{definition}

Thomassen~\cite{Thomassen1988} proved that every large enough wall also contains a large clean subwall. 

\begin{proposition}[\cite{Thomassen1988}]
    \label{prop:cleanwall}
    There exists a computable function~${h_{\ref{prop:cleanwall}} \colon \mathbb{N} \to \mathbb{N}}$ such that for every positive integer~$k$, every ${h_{\ref{prop:cleanwall}}(k)}$-wall contains a clean $k$-wall as a subgraph. 
\end{proposition}

\begin{proposition}[follows from~\cites{Bodlaender1996,Diestel1999}]
    \label{prop:gridtheorem}
    There exist computable functions ${f_{\ref{prop:gridtheorem}},g_{\ref{prop:gridtheorem}} \colon \mathbb{N} \to \mathbb{N}}$ and an algorithm that, given a positive integer~${k \geq 3}$ and a graph~$G$ as inputs, 
    in time ${\mathcal{O}(g_{\ref{prop:gridtheorem}}(k) \abs{V(G)})}$, 
    either finds a $k$-wall in~$G$, 
    or concludes correctly that~${\mathsf{tw}(G) \leq f_{\ref{prop:gridtheorem}}(k)}$. 
\end{proposition}

To obtain a grid theorem for \oddminors, and a wall variant, we now combine the two theorems above with \cref{prop:bipartite-odd-minor-equiv}.

\begin{corollary}
    \label{cor:cleanwall}
    There exist computable functions ${f_{\ref{cor:cleanwall}},g_{\ref{cor:cleanwall}} \colon \mathbb{N} \to \mathbb{N}}$ and an algorithm that, given a positive integer~${k \geq 3}$ and a graph~$G$ as inputs, 
    in time~${\mathcal{O}(g_{\ref{cor:cleanwall}}(k) \abs{V(G)})}$, 
    either finds a clean $k$-wall in~$G$, 
    or concludes correctly that~${\mathsf{tw}(G) \leq f_{\ref{cor:cleanwall}}(k)}$. 
\end{corollary}

\begin{corollary}
    \label{cor:oddgridtheorem}
    There exists a computable function~${f_{\ref{cor:oddgridtheorem}} \colon \mathbb{N} \to \mathbb{N}}$ such that for every positive integer~$k$, every graph with~${\mathsf{tw}(G) \geq f_{\ref{cor:oddgridtheorem}}(k)}$ contains the ${(k\times k)}$-grid as an \oddminor. 
\end{corollary}

\subsection{%
\texorpdfstring{Single-crossing-minor-free graphs and $\mathcal{P}$-blind-treewidth}%
{Single-crossing-minor-free graphs and P-blind-treewidth}}

Given some family of non-planar and non-bipartite graph~$\mathcal{H}$ and a graph~$G$ which does not contain any member of~$\mathcal{H}$ as an \oddminor.
One possibility for~$G$ could be that~$G$ does not contain any member of~$\mathcal{H}$ even as a minor. 
In this situation, we lose any control over the structure of odd cycles in~$G$. 
However, in this case we enter the realm of $\mathcal{H}$-minor-free graphs and, if a structure theorem for such graphs is known, we may apply this instead of analysing the structure of~$G$ with respect to parity. 
One instance where this will occur in this paper is the exclusion of parity-crossing graphs as \oddminors. 

Consider the following parametric graph. 
For~${k \in \mathbb{Z}^{+}}$ let~$\mathscr{U}_k$ be the graph obtained from the ${(2k\times 2k)}$-grid by adding the edges~${\{ (k,k), (k+1,k+1) \}}$ and~${\{ (k+1,k), (k,k+1) \}}$. 
We call~${\mathscr{U} = \langle \mathscr{U}_k \rangle_{k \in \mathbb{Z}^{+}}}$ the \emph{single-crossing grid} and the graph~$\mathscr{U}_k$ the single-crossing grid of order~$k$. 
See \Cref{figure:singlecrossinggrid} for an illustration.
A graph is called \emph{singly crossing} if it can be embedded in the plane with a single crossing and a graph is a \emph{single-crossing minor} if it is a minor of some singly crossing graph. 
Clearly, $\mathscr{U}_k$ is singly crossing. 
Moreover, for every single-crossing minor~$H$ there exists some~${k_H \in \mathbb{Z}^{+}}$ such that~${H \leq_{\mathsf{minor}} \mathscr{U}_{k_h}}$. 

We now discuss the single-crossing-minor-free structure theory in the context of the parameter $\mathcal{P}$-blind-treewidth.

\begin{lemma}
    \label{lem:P-tw-single-crossing-grid}
    It holds that ${\mathcal{P}\text{-}\mathsf{blind}\text{-}\mathsf{tw}(\mathscr{U}_k) = \mathsf{tw}(\mathcal{U}_k) + 1}$ for every integer~${k \geq 3}$. 
\end{lemma}

\begin{proof}
    Fix an integer~${k \geq 3}$. 
    Clearly, ${\mathcal{P}\text{-}\mathsf{blind}\text{-}\mathsf{tw}(\mathscr{U}_k) \leq \mathsf{tw}(\mathscr{U}_k) + 1}$. 
    Suppose for a contradiction that ${k' \coloneqq \mathcal{P}\text{-}\mathsf{blind}\text{-}\mathsf{tw}(\mathscr{U}_k) \leq \mathsf{tw}(\mathcal{U}_k)}$. 
    Let~$\mathcal{T} = (T,\beta)$ be a tree-decomposition of~$\mathscr{U}_k$ of $\mathcal{P}$-blind-width~${k'}$. 
    Since by assumption~${k' \leq \mathsf{tw}(\mathscr{U}_k)}$, there is a node~${t \in V(T)}$ such that~${\abs{\beta(t)} > k'}$ and the torso~$H$ on~$t$ is planar. 
    
    We may assume without loss of generality that for every adhesion set~$S$ of~$t$, every vertex in~$S$ has a neighbour in~${\beta(t) \setminus S}$. 
    First observe that, since~$H$ is planar, the adhesion of~$t$ is at most~$4$, since otherwise~$H$ would contain a subgraph isomorphic to~$K_5$. 
    Moreover, no adhesion set is a minimal separator of size~$4$, since otherwise~$H$ would contain a $K_5$-minor. 
    Hence, $\beta(t)$ contains~${\{ (i,j) \mid i,j \in [2,2k-1] \}}$. 
    But then, contracting each edge of~${H - \{(k,k), (k+1,k), (k,k+1), (k+1,k+1)\}}$ yields a $K_5$-minor in~$H$. 
\end{proof}

\begin{corollary}
    \label{cor:pb-f(P-tw)}
    It holds that~${\mathsf{big}_{\mathscr{U}}(G) \leq (\mathcal{P}\text{-}\mathsf{blind}\text{-}\mathsf{tw}(G) - 1)/2 + 2}$ for every graph~$G$. 
\end{corollary}

\begin{proof}
    Let~$G$ be a graph with~${\mathsf{big}_{\mathscr{U}}(G) = k}$ for some positive integer~$k$. 
    If~${k \leq 2}$, the statement clearly holds. 
    Since treewidth is $\leq_{\mathsf{minor}}$-monotone and planarity is closed under minors, we observe that~$\mathcal{P}\text{-}\mathsf{blind}\text{-}\mathsf{tw}$ is $\leq_{\mathsf{minor}}$-monotone. 
    Using \Cref{lem:P-tw-single-crossing-grid}, we calculate 
    \[
        \mathcal{P}\text{-}\mathsf{blind}\text{-}\mathsf{tw}(G) \geq \mathcal{P}\text{-}\mathsf{blind}\text{-}\mathsf{tw}(\mathscr{U}_k) = \mathsf{tw}(\mathscr{U}_k)+1 \geq \mathsf{tw}(\mathscr{G}_{2k})+1 = 2k + 1 = 2\mathsf{big}_{\mathscr{U}}(G) + 1.\qedhere
    \]
\end{proof}

\begin{proposition}[Follows from \cite{RobertsonS1993,GiannopoulouTW23-singlecrossing,Grohe2016,Bodlaender1996}]
    \label{prop:singlecrossing}
    It holds that ${\mathcal{P}\text{-}\mathsf{blind}\text{-}\mathsf{tw} \sim \mathsf{big}_{\mathscr{U}}}$. 
    In particular, there exist computable functions~${f_{\ref{prop:singlecrossing}}, g_{\ref{prop:singlecrossing}} \colon \mathbb{N} \to \mathbb{N}}$ and an algorithm that, given a positive integer~$k$ and a graph~$G$ as inputs, in time $\mathcal{O}(g_{\ref{prop:singlecrossing}}(k) \abs{V(G)}^3)$, finds either
    \begin{enumerate}
        \item an expansion of the single-crossing grid of order~$k$, or
        \item a tree-decomposition~${(T,\beta)}$ of $\mathcal{P}$-blind-width~$f_{\ref{prop:singlecrossing}}(k)$, where, in particular, for every node~${t \in V(T)}$ with~${\abs{\beta(t)} > f_{\ref{prop:singlecrossing}}(k)}$, we have that the torso of~$G$ on the node~$t$ is a planar quasi-$4$-component of~$G$ and the adhesion of~$t$ is at most~$3$. 
    \end{enumerate}
\end{proposition}

Robertson and Seymour's original work~\cite{RobertsonS1993} proved the existence of the tree-decomposi-tion in single-crossing-minor-free graphs. 
An algorithmic version like this can be obtained by first applying \cref{prop:4components} and refining this tree-decomposition by applying \Cref{prop:treewidth,prop:gridtheorem} to each quasi-$4$-component whose torso is non-planar. 
In the case that we find the wall, we can use Corollary 2.13 from the arXiv version of~\cite{GiannopoulouTW23-singlecrossing} to find the $\mathscr{U}_k$-expansion. 
With \Cref{cor:pb-f(P-tw)}, the asymtotic equivalence of the parameters follows. 

\section{\Oddminors\ and bipartite graphs}
\label{sec:bipartite}

We begin by making some basic observations on the behaviour of the \oddminor\ relation in bipartite graphs.

\begin{proposition}
    \label{prop:bipartite-odd-minor-equiv}
    A bipartite graph~$H$ is a minor of a bipartite graph~$G$ if and only if~$H$ is an \oddminor\ of~$G$. 
\end{proposition}

\begin{proof}
    Clearly if~$H$ is an \oddminor\ of~$G$, then~$H$ is a minor of~$G$. 

    For the other direction,
    let~${c_G \colon V(G) \to [2]}$ denote a proper $2$-colouring of~$G$ and let~${c_H \colon V(H) \to [2]}$ denote a proper $2$-colouring of~$H$. 
    We define~${c \colon V(H_\eta) \to [2]}$ by setting~${c(x) = c_G(v)}$ if~${x \in \eta(v)}$ for some~${v \in V(H)}$ with~${c_H(v) = 1}$ and~${c(x) = 2 - c_G(x)}$, otherwise.  
    By definition, it is straight forward to check that~$c$ witnesses that~$\eta$ is odd.  
\end{proof}

\begin{corollary}
    \label{cor:bipartite-subdivision}
    Let $H$ be a bipartite graph.
    If a graph $G$ contains a bipartite subdivision of $H$ as a subgraph, then it contains $H$ as an \oddminor.
\end{corollary}

\begin{proof}
    Suppose $G$ contains the bipartite subdivision $H'$ of $H$ as a subgraph.
    Then clearly, it contains $H'$ as an \oddminor.
    Since $H$ and $H'$ are both bipartite and $H'$ contains $H$ as a minor, the assertion is directly implied by \cref{prop:bipartite-odd-minor-equiv} and the transitivity of the \oddminor\ relation.
\end{proof}

\begin{proposition}
    \label{prop:bipartite-odd-minor-closed}
    If a graph~$H$ is an \oddminor\ of a bipartite graph~$G$, then~$H$ is bipartite. 
\end{proposition}

\begin{proof}
    Let~${c_G \colon V(G) \to [2]}$ be a proper $2$-colouring of~$G$, 
    let~$\eta$ be an odd $H$-expansion in~$G$, 
    and let~$c$ witness that~$\eta$ is odd. 
    We define~${c_H \colon V(H) \to [2]}$ by setting~${c_H(v) = 1}$ if~${c(x) = c_G(x)}$ for each~${x \in \eta(v)}$, and~${c_H(v) = 0}$, otherwise. 
    For each edge~${vw \in E(H)}$, since~$c(v) = c(w)$ and~$c_G$ is proper, we obtain that~${c_H(v) \neq c_H(w)}$. 
    Hence, $c_H$ is proper. 
\end{proof}

\begin{corollary}
    \label{cor:oddmonotone}
    The graph parameters~$\mathcal{B}\text{-}\mathsf{blind}\text{-}\mathsf{tw}$ and~$(\mathcal{B} \cup \mathcal{P})\text{-}\mathsf{blind}\text{-}\mathsf{tw}$ are $\odd$-monotone. 
\end{corollary}

\begin{proof}
    This follows from the facts that~$\mathsf{tw}$ is $\leq_{\mathsf{minor}}$-monotone, that planarity is closed under minors, and from \Cref{prop:bipartite-odd-minor-closed}. 
\end{proof}

\section{Odd ears}
\label{sec:ears}

Let~$H$ be a subgraph of a graph~$G$. 
An \emph{$H$-ear (in~$G$)} is a path in~$G$ whose endvertices are in~$V(H)$, which is internally disjoint from~$H$ and edge-disjoint from~$H$. 
We call an~$H$-ear~$P$ \emph{odd} if~$H$ is bipartite and~${H \cup P}$ is non-bipartite. 

\begin{lemma}
    \label{lemma:oddear}
    Let~$H$ be a $2$-connected bipartite subgraph of a $2$-connected non-bipartite graph~$G$. 
    Then there is a odd $H$-ear in~$G$, which can be found in time~$\mathcal{O}(\abs{V(G)})$.  
\end{lemma}

\begin{proof}
    Let~$c_H$ be a proper $2$-colouring of~$H$ and let~$C$ be an odd cycle in~$G$, which can be found in linear time by greedily trying to extend the proper $2$-colouring of~$H$ to a proper $2$-colouring of~$G$. 
    Suppose~${\abs{V(C) \cap V(H)} \leq 1}$. 
    Since~$G$ is $2$-connected, we find two (possibly trivial) vertex-disjoint $(V(C),V(H)$-paths~$Q_1$ and~$Q_2$ in linear time. 
    In particular, if ${\abs{V(C) \cap V(H)} = 1}$, one of~$Q_1$ or~$Q_2$ is trivial. 
    Note that~${Q_1 \cup C \cup Q_2}$ contains exactly two $H$-ears~$P_1$ and~$P_2$. 
    Since \[{\abs{E(P_1)} + \abs{E(P_2)} = \abs{E(C)} + 2 (\abs{E(Q_1)} + \abs{E(Q_2)})},\] we obtain that~$P_1$ and~$P_2$ have different parities. 
    And since both~$P_1$ and~$P_2$ have the same set of endvertices in~$V(H)$, one of them is an odd $H$-ear. 
    
    So we may assume that~${\abs{V(C) \cap V(H)} \geq 2}$. 
    Let~$\mathcal{P}$ denote the set of~$H$-ears in~${H \cup C}$, and let~$\mathcal{Q}$ denote the set of maximal paths in~${H \cap C}$. 
    Observe that~$C$ is the edge-disjoint union of all paths in~$\mathcal{P} \cup \mathcal{Q}$. 
    For~${i \in [2]}$, let~${X_i \coloneqq V(\bigcup \mathcal{P}) \cap c_H^{-1}(i)}$ denote the set of all vertices~$v$ in an $H$-ear in~$\mathcal{P}$ with~${c_H(v) = i}$. 
    Without loss of generality, we may assume that~${\abs{X_1} \geq 1}$. 
    Assume for a contradiction that no $H$-ear in~$\mathcal{P}$ is odd. 
    Then every~${(X_1,X_2)}$-path in~$\mathcal{P} \cup \mathcal{Q}$ is odd and every ${(X_i,X_i)}$-path in~$\mathcal{P} \cup \mathcal{Q}$ for any~${i \in [2]}$ is even. 
    Now it is easy to observe that that the number of~${(X_1,X_2)}$-path in~$\mathcal{P} \cup \mathcal{Q}$ is even, contradicting that~$C$ is odd.  
\end{proof}

\section{Excluding a singly parity-breaking grid as an \oddminor}
\label{sec:singleodd}

In this section we prove \cref{thm:singleoddstructure}. 
Recall the definition of singly parity-breaking grids, see \Cref{figure:singleoddintro}.

\begin{definition}[Singly parity-breaking grids]
    \label{def:singleodd}
    The \emph{singly parity-breaking grid}~$\mathscr{S}_k$ of \emph{order~$k$} is the graph obtained from the ${(2k\times 2k)}$ grid by adding the edge~${\{(k,k),(k+1,k+1)\}}$.
    We denote the parametric singly parity-breaking grid by~${\mathscr{S} \coloneqq \langle \mathscr{S}_k \rangle_{k \in \mathbb{Z}^{+}}}$. 
\end{definition}

\begin{lemma}
    \label{lem:B-tw-singly-grid}
    It holds that ${\mathcal{B}\text{-}\mathsf{blind}\text{-}\mathsf{tw}(\mathscr{U}_k) = \mathsf{tw}(\mathcal{U}_k) + 1}$ for every integer~${k \geq 1}$. 
\end{lemma}

\begin{proof}
    Fix an integer~${k \geq 1}$. 
    Clearly, ${\mathcal{B}\text{-}\mathsf{blind}\text{-}\mathsf{tw}(\mathscr{S}_k) \leq \mathsf{tw}(\mathscr{S}_k) + 1}$. 
    Suppose for a contradiction that ${k' \coloneqq \mathcal{B}\text{-}\mathsf{blind}\text{-}\mathsf{tw}(\mathscr{S}_k) \leq \mathsf{tw}(\mathscr{S}_k)}$. 
    Let~${\mathcal{T} = (T,\beta)}$ be a tree-decomposition of~$\mathscr{S}_k$ of $\mathcal{B}$-blind-width~${k'}$. 
    By assumption it holds that~${k' \leq \mathsf{tw}(\mathscr{S}_k)}$, there is a node~${t \in V(T)}$ such that~${\abs{\beta(t)} \geq k}$ and~${(G,\beta(t)) \in \mathcal{B}}$.

    Now notice that for every choice of vertices $u,v\in V(\mathscr{S}_k)$ there exists an odd cycle $C_{uv}$ in $\mathscr{S}_k$ containing both $u$ and $v$.
    Since $\abs{\beta(t)}\geq k$ we may choose $u,v\in \beta(t)$.
    The existence of the cycle $C_{uv}$ now certifies that $\beta(t)$ cannot be globally bipartite, a contradiction.
\end{proof}

\begin{corollary}
    \label{cor:bpb-f(B-tw)}
    It holds that~${\mathsf{big}_{\mathscr{S}}(G) \leq (\mathcal{B}\text{-}\mathsf{blind}\text{-}\mathsf{tw}(G) - 1)/2}$ for every graph~$G$.
\end{corollary}

\begin{proof}
    Let~$G$ be a graph with~${\mathsf{big}_{\mathscr{S}}(G) = k}$ for some positive integer~$k$. 
    With \Cref{cor:oddmonotone} and \Cref{lem:B-tw-singly-grid}, we calculate 
    \[
        \mathcal{B}\text{-}\mathsf{blind}\text{-}\mathsf{tw}(G) 
        \geq \mathcal{B}\text{-}\mathsf{blind}\text{-}\mathsf{tw}(\mathscr{S}_k) 
        = \mathsf{tw}(\mathscr{S}_k)+1 
        \geq \mathsf{tw}(\mathscr{G}_{2k})+1 
        = 2k + 1
        = 2\mathsf{big}_{\mathscr{S}}(G) + 1.\qedhere
    \]
\end{proof}

Our first goal is to show that a non-bipartite $2$-connected graph of large treewidth contains a singly parity-breaking grid as an odd-minor. 
To do this, we need the following two lemmas, the first of which allows us to turn a wall ``inside out''. 

\begin{lemma}
    \label{lem:InsideOutWall}
    There is a computable function~${f_{\ref{lem:InsideOutWall}} \colon \mathbb{N} \to \mathbb{N}}$ satisfying the following. 
    For every positive integer~$k$, every $f_{\ref{lem:InsideOutWall}}(k)$-wall~$W$ 
    contains a $2k$-wall~$W'$ as a subgraph such that 
    \begin{itemize}
        \item the perimeter of~$W$ is equal to the central brick of~$W'$, and 
        \item the four corners of~$W$ are degree~$2$ vertices in the central brick of~$W'$.
    \end{itemize}
    Moreover, if~$W$ is clean, then~$W'$ is clean as well, and vertices in~$W'$ corresponding to the four corners of~$W$ are in the same colour class as the branch vertices of~$W'$ for every proper $2$-colouring of~$W$. 
\end{lemma}

We do not provide a formal proof for this lemma but refer the reader to \Cref{figure:insideout}.
In the figure we provide a scheme for finding a subgraph within a $14$-wall that we can turn inside out to obtain a $6$-wall that fixes the positions of the corners as desired. 

\begin{figure}[ht]
    \begin{center}
        \includegraphics[scale=0.6]{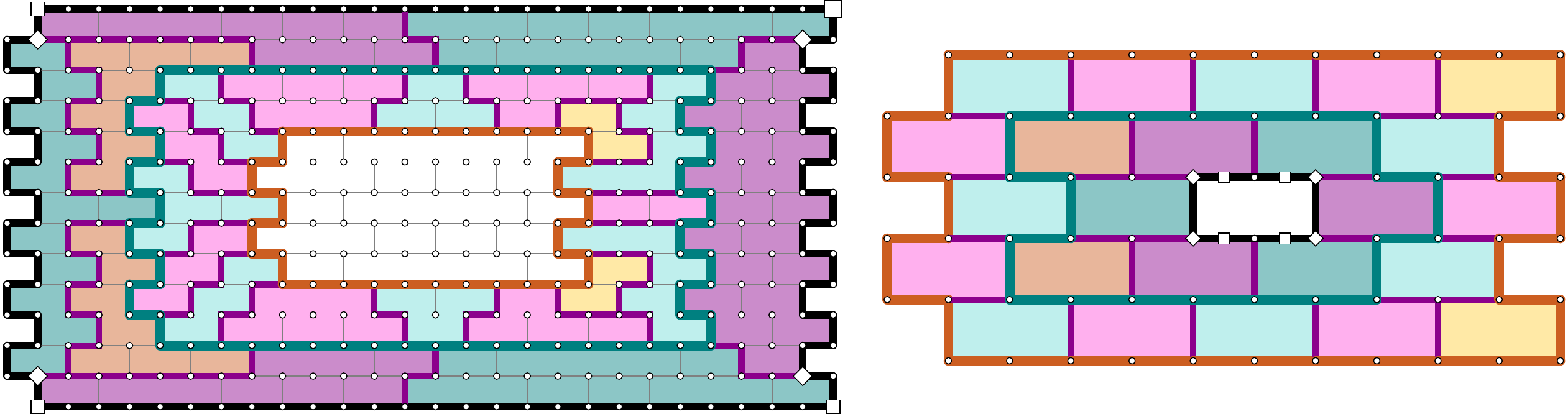}
    \end{center}
    \caption{The process of turning a a $6$-wall within a $14$-wall inside out. If every drawn edge stands for a path of even length, cleanness is preserved. }
    \label{figure:insideout}
\end{figure}

The second auxiliary lemma constructs a singly parity-breaking grid in a large lean wall with an odd ear attached to its perimeter. 

\begin{lemma}
    \label{lem:cleanAndOddHandle2singlyoddminor}
    Let~$k$ be a positive integer, let~$W$ be a clean $(f_{\ref{lem:InsideOutWall}}(k)+2)$-wall and let~$P$ be an odd $W$-ear in~${W \cup P}$ whose endvertices are in the perimeter of~$W$. 
    Then~${W \cup P}$ contains~$\mathscr{S}_k$ as an \oddminor. 
\end{lemma}

\begin{proof}
    Let~$C$ denote the perimeter of~$W$,
    let~$W'$ be the subwall of $W$ of order~$f_{\ref{lem:InsideOutWall}}(k)$ obtained by first deleting~$C'$ and then successively deleting vertices of degree~$1$, and let~$C'$ denote the perimeter of~$W'$. 
    By \Cref{lem:InsideOutWall} there is a clean $2k$-wall ${W'' \subseteq W'}$ whose central brick is bounded by~$C_1$.
    Now since~${C \cup P}$ is $2$-connected, there an odd $W''$-ear between any two corners of~$W''$. 
    Let~$P'$ denote an odd $W''$-ear between the vertices corresponding to the vertices~${(2k,k)}$ and~${(2k+2,k+1)}$ of the elementary $2k$-wall. 
    By contracting every path corresponding to every other edge in each row of~$W''$, we obtain the $(2k,2k)$-grid~$H$ as an \oddminor\ of~$W''$ by \Cref{prop:bipartite-odd-minor-equiv} as witnessed by some odd $H$-expansion~$\eta'$. 
    Define an~$\mathscr{S}_k$-expansion~$\eta$ in~${W \cup P}$ as follows. 
    Let~$P''$ denote the subpath of~$P'$ obtained by deleting the vertex corresponding to~${(2k+2,k+1)}$ and let~$e$ denote unique edge in~$E(P') \setminus E(P'')$. 
    We set~${\eta((k,k)) \coloneqq (\eta'((k,k)) \cup P'')}$, ${\eta(\{ (k,k), (k+1,k_1) \}) \coloneqq e}$, and~${\eta(x) \coloneqq \eta'(x)}$, otherwise. 
    Let~$c'$ be the $2$-colouring of~${H_\eta'}$ that witnesses that~$\eta'$ is odd, and let~$c$ be the unique $2$-colouring of~${(\mathscr{S}_k)_\eta}$ for which~${c \restricted \eta((k,k))}$ is a proper $2$-colouring. 
    Since~$P''$ is even, both of its endvertices have the same colour, say colour~$1$. 
    Since~$W''$ is clean, each path between the endvertices of~$P'$ in~$W''$ is even. 
    Since $\eta'$ is odd, each such path contains an even amount of monochromatic edges as well as an even amount of bichromatic edges. 
    Hence, $\eta(e)$ is monochromatic, and~$c$ witnesses that~$\eta$ is odd, as desired. 
\end{proof}

Using these two lemmas, we prove the following important step. 

\begin{lemma}
    \label{lem:find-spb}
    There exists computable functions~${f_{\ref{lem:find-spb}}, g_{\ref{lem:find-spb}}}$ and an algorithm that, given a positive integer~$k$ and a non-bipartite $2$\nobreakdash-connected graph~$G$ and as inputs, 
    in time \linebreak ${\mathcal{O}(g_{\ref{lem:find-spb}}(k)\abs{V(G)}\log\abs{V(G)})}$,
    either finds an odd $\mathscr{S}_k$-expansion in~$G$, 
    or concludes correctly that ${\mathsf{tw}(G) < f_{\ref{lem:find-spb}}(\mathsf{big}_{\mathscr{S}}(G))}$. 
\end{lemma}

\begin{proof}
    Let $f_{\ref{cor:cleanwall}}$ be the function from \Cref{cor:cleanwall} and let~$f_{\ref{lem:InsideOutWall}}$ be the function from \Cref{lem:InsideOutWall}. 
    Let~${f_{\ref{lem:find-spb}}(k) \coloneqq f_{\ref{cor:cleanwall}}(2(f_{\ref{lem:InsideOutWall}}(k)+2))}$. 
    We may assume that $f_{\ref{cor:cleanwall}}$ and~$f_{\ref{lem:InsideOutWall}}$ are monotone, and hence that~$f_{\ref{lem:find-spb}}$ is monotone as well. 

    Assume~${\mathsf{tw}(G) > f(k)}$. 
    By \Cref{cor:cleanwall}, $G$ contains a clean ${2(f_{\ref{lem:InsideOutWall}}+2)}$-wall~$W$ as a subgraph (found in linear time depending on~$k$). 
    By \Cref{lemma:oddear}, $G$ contains odd $W$-ear~$P$ (found in linear time). 
    Now since the order of~$W$ is~${2(f_{\ref{lem:InsideOutWall}}(k)+2)}$, it follows that~$W$ contains a ${(f_{\ref{lem:InsideOutWall}}(k)+2)}$-wall~$W'$ as a subgraph that is vertex-disjoint from~$P$. 
    Since~$W$ is $2$-connected, $G'$ contains an odd $W'$-ear~$P'$ whose endvertices are in the perimeter of~$W'$. 
    Now \Cref{lem:cleanAndOddHandle2singlyoddminor} yields an odd $\mathscr{S}_k$-expansion (found constant time depending on~$k$). 
\end{proof}

It is well known that the tree-decomposition from \Cref{prop:cutvertextree} of a graph~$G$, can be refined with arbitrary tree-decompositions of the blocks of~$G$. 
Hence, any width parameter that is monotone with respect to the blocks of a graph is equal to the maximum width over all its blocks. 
In particular, we observe the following. 
\begin{observation}
    \label{obs:blind-tw-blocks}
    The $\mathcal{G}$-blind-treewidth of a graph~$G$ is equal to the maximum $\mathcal{G}$-blind-treewidth over all blocks of~$G$ for $\mathcal{G}$ equal to either~$\mathcal{B}$ or~$(\mathcal{B} \cup \mathcal{P})$. 
\end{observation}

Finally, we prove \Cref{thm:singleoddstructure} in a slightly stronger version. 

\begin{theorem}\label{thm:singleoddstructure-actual}
    It holds that ${\mathcal{B}\text{-}\mathsf{blind}\text{-}\mathsf{tw} \sim \mathsf{big}_{\mathscr{S}}}$. 
    In particular, there exist computable functions ${f_{\ref{thm:singleoddstructure-actual}},g_{\ref{thm:singleoddstructure-actual}}\colon\mathbb{N}\to\mathbb{N}}$ 
    and an algorithm that, given a positive integer~$k$ and a graph~$G$ as inputs, 
    in time $\mathcal{O}(g_{\ref{thm:singleoddstructure-actual}}(k)(\abs{V(G)} \log \abs{V(G)} + \abs{E(G)}))$, 
    finds either
    \begin{enumerate}
        \item an odd expansion of the singly parity-breaking grid of order~$k$, or
        \item a tree-decomposition~${(T,\beta)}$ for~$G$ of $\mathcal{B}$-blind-width~$f_{\ref{thm:singleoddstructure-actual}}(k)$, where, in particular, for every node~${t \in V(T)}$ with~${\abs{\beta(t)} > f_{\ref{thm:singleoddstructure-actual}}(k)}$, we have that~$G[\beta(t)]$ is a bipartite block of~$G$ and the adhesion of~$t$ is at most one.
    \end{enumerate}
\end{theorem}

\begin{proof}
    Let~$f_{\ref{thm:singleoddstructure-actual}} \coloneqq f_{\ref{lem:find-spb}}$ from \Cref{lem:find-spb}. 
    We use \Cref{prop:cutvertextree} and for each block \Cref{prop:treewidth} to either find a non-bipartite block~$B$ with~${\mathsf{tw}(B) > f_{\ref{thm:singleoddstructure-actual}}(k)}$ and an odd $\mathscr{S}_k$-expansion by \Cref{lem:find-spb}, 
    or a tree-decomposition of $\mathcal{B}$-blind-width~$f(k)$. 
    
    With \Cref{cor:bpb-f(B-tw)}, the asymptotic equivalence of the parameters follows. 
\end{proof}

\section{Excluding a singly parity-crossing grid as an \oddminor}
\label{sec:oddcrossing}

In this section we prove \cref{thm:singleoddcrossstructure}. 
To do this, we roughly follow the same strategy as we did in \Cref{sec:singleodd}.
However, since the structure we are interested in is slightly more complicated, we have to perform some additional intermediate steps. 
That is, we first show that every $2$-connected non-bipartite graph that cannot be decomposed along separators of size at most three into torsos which are either small or planar contains a clean wall with a cross on its four corners as a subgraph. 
Once this is established, we make use of \cref{lemma:oddear} as before to establish a parity break which will then be combined with the cross to form one of the three singly parity-crossing grids as an \oddminor.

We begin with a precise definition of the singly parity-crossing grids, see \Cref{figure:singleoddcross}. 

\begin{definition}[Singly parity-crossing grids]
    \label{def:singleoddcross}
    The graph~$\mathscr{C}^1_k$ is the graph obtained from the ${(2k \times 2k)}$-grid by adding the edges~${\{(k,k),(k+1,k+1)\}}$ and~${\{(k,k+1),(k+1,k)\}}$. 
    We denote the corresponding parametric graph by~${\mathscr{C}^1 \coloneqq \langle \mathscr{C}^1_k \rangle_{k \in \mathbb{Z}^{+}}}$. 
    
    The graph~$\mathscr{C}^2_k$ is the graph obtained from the ${(2k \times 2k)}$-grid by adding the edges ${\{(k,k),(k+1,k+1)\}}$, ${\{(k,k+1),(k+1,k)\}}$, and subdividing~${\{(k,k+1),(k+1,k)\}}$ with a vertex we call~$x$. 
    We denote the corresponding parametric graph by~${\mathscr{C}^2 \coloneqq \langle \mathscr{C}^2_k \rangle_{k \in \mathbb{Z}^{+}}}$. 
    
    The graph~$\mathscr{C}^3_k$ is the graph obtained from the ${(2k \times 2k)}$-grid by adding the edges ${\{(k,k),(k+1,k+1)\}}$ and~${\{(k,k+1),(k+1,k)\}}$, subdividing both of them with vertices~$y$ and~$x$, respectively, and adding an edge between~${(k+1,k+1)}$ and~$x$. 

    For each~${i \in [3]}$, we say~$\mathscr{C}^i_k$ is a \emph{singly-parity-crossing grid of order~$k$}. 
    
    We denote the corresponding parametric graph by~${\mathscr{C}^3 \coloneqq \langle \mathscr{C}^3_k \rangle_{k \in \mathbb{Z}^{+}}}$. 
\end{definition}

\begin{lemma}
    \label{lem:BP-tw-singly-crossing-grid}
    It holds that ${(\mathcal{B} \cup \mathcal{P})\text{-}\mathsf{blind}\text{-}\mathsf{tw}(\mathscr{C}^i_k) = \mathsf{tw}(\mathcal{U}_k) + 1}$ for every~${i \in [3]}$ and every integer~${k \geq 3}$. 
\end{lemma}

\begin{proof}
    Fix~${i \in [3]}$ and an integer~${k \geq 3}$. 
    Clearly, ${(\mathcal{B} \cup \mathcal{P})\text{-}\mathsf{blind}\text{-}\mathsf{tw}(\mathscr{C}^i_k) \leq \mathsf{tw}(\mathscr{C}^i_k) + 1}$. 
    Suppose for a contradiction that ${k' \coloneqq (\mathcal{B} \cup \mathcal{P})\text{-}\mathsf{blind}\text{-}\mathsf{tw}(\mathscr{C}^i_k) \leq \mathsf{tw}(\mathscr{C}^i_k)}$. 
    Let~$\mathcal{T} = (T,\beta)$ be a tree-decomposition of~$\mathscr{C}^i_k$ of $(\mathcal{B}\cup\mathcal{P})$-blind-width~${k'}$.
    Since by assumption~${k' \leq \mathsf{tw}(\mathscr{C}^i_k)}$, there is a node~${t \in V(T)}$ such that~${\abs{\beta(t)} > k'}$ and either $\beta(t)$ is globally bipartite or the torso~$H$ on~$t$ is planar. 

    Similar to the proof of \cref{lem:B-tw-singly-grid} we may observe that in any of the three singly parity-crossing grid of order $k$, any pair of vertices is contained in an odd cycle.
    Hence, we may further assume that $\beta(t)$ is not globally bipartite. 
    
    As in the proof of \cref{lem:P-tw-single-crossing-grid}, notice that~$\beta(t)$ contains~${\{ (i,j) \mid i,j \in [2,2k-1] \}}$. 
    If~${i = 1}$, this immediately yields a~$K_5$-minor in~$H$ as in the proof of \cref{lem:P-tw-single-crossing-grid}, a contradiction. 
    If~${i = 2}$, then note that if~${x \notin \beta(t)}$, then $\{ (k,k+1), (k+1,k) \}$ is an adhesion set of~$t$. 
    So regardless whether~${x \in \beta(t)}$, we get the $K_5$-minor as in the previous case. 
    Finally, if~${i = 3}$, then yields similarly as in the previous cases a $K_5$-minor, whether~$x$ or~$y$ are contained in~$\beta(t)$ or not. 
\end{proof}

\begin{corollary}
    \label{cor:bpb-f(BP-tw)}
    It holds that~${\mathsf{big}_{\mathfrak{C}}(G) \leq ((\mathcal{B}\cup \mathcal{P})\text{-}\mathsf{blind}\text{-}\mathsf{tw}(G) - 1)/2 + 2}$ for every graph~$G$.
\end{corollary}

\begin{proof}
    Let~$G$ be a graph with~${\mathsf{big}_{\mathfrak{C}}(G) = k}$ for some positive integer~$k$. 
    If~${k \leq 2}$, the statement clearly holds. 
    With \Cref{cor:oddmonotone} and \Cref{lem:B-tw-singly-grid}, we calculate for each~${i \in [3]}$, 
    \[
        (\mathcal{B} \cup \mathcal{P})\text{-}\mathsf{blind}\text{-}\mathsf{tw}(G) 
        \geq (\mathcal{B} \cup \mathcal{P})\text{-}\mathsf{blind}\text{-}\mathsf{tw}(\mathscr{C}^i_k) 
        = \mathsf{tw}(\mathscr{C}^i_k) + 1 
        \geq \mathsf{tw}(\mathscr{G}_{2k}) + 1 
        = 2\mathsf{big}_{\mathscr{S}}(G) + 1, 
    \]
    as desired. 
\end{proof}

Next, we show that every graph that contains a huge single-crossing grid as a minor must contain a large clean wall with a cross on its corners as a subgraph.

Let~${r}$ be a positive integer, let~$W$ be a clean $r$-wall, and let~$s_1$, $s_2$, $t_1$, and~$t_2$ be its corners named, in this cyclic order, after their appearance when traversing along the perimeter of~$W$. 
We say that the graph~$U$ obtained from~$W$ by adding two vertex-disjoint $W$-ears~$P_1$, $P_2$ such that~$P_i$ has endpoints~$s_i$ and~$t_i$ for each~${i \in [2]}$, is a \emph{clean cross-wall} of \emph{order~$r$}. 

\begin{lemma}
    \label{lemma:cleancrosswall}
    There exists a computable function~${f_{\ref{lemma:cleancrosswall}} \colon \mathbb{N} \to \mathbb{N}}$ such that for every~${k \in \mathbb{N}}$ and every graph~$G$ that contains the single-crossing grid of order~${f_{\ref{lemma:cleancrosswall}}(k)}$ as a minor, $G$ also contains a clean cross-wall of order~$k$ as a subgraph. 
\end{lemma}

\begin{proof}
    Let~$h_{\ref{prop:cleanwall}}$ be the function from \cref{prop:cleanwall}. 
    We set~${f_{\ref{lemma:cleancrosswall}}(k) \coloneqq 2(h_{\ref{prop:cleanwall}}(k)+2)+2}$. 
    
    Consider~${H_1 \coloneqq \mathscr{G}_{2(h_{\ref{prop:cleanwall}}(2k)+1)+2}}$ as a subgraph of ${\mathscr{U}_{2(h_{\ref{prop:cleanwall}}(2k)+1)+2}}$. 
    Now, let~$H_2$ be the subgraph of $H_1$ by deleting all vertices of $H_1$ that belong to the unique face of~$H_1$ which is not a $4$-cycle.
    Observe that~$H_2$ is isomorphic to~$\mathscr{G}_{2(h_{\ref{prop:cleanwall}}(k)+2)}$ 
    and thus, it contains a subgraph~$H_3$ isomorphic to the elementary ${(h_{\ref{prop:cleanwall}}(k)+2)}$-wall as a subgraph. 
    
    Let~$s_1'$, $s_2'$, $t_1'$, $t_2'$ be the four corners of~$H_3$ as encountered while traversing along its perimeter. 
    By using the cross of~$\mathscr{U}_{2(h_{\ref{prop:cleanwall}}(k)+2)+2}$ we may find the two vertex-disjoint $H_3$-ears~$P_1'$ and~$P_2'$ such that~$P_i'$ has endpoints~$s_i'$ and~$t_i'$ for each~${i \in [2]}$. 
    Let~$H_4$ be the union of~$H_3$, $P_1'$, and~$P_2'$, and let~$H_4'$ be the ${h_{\ref{prop:cleanwall}}(k)}$-subwall of~$H_3$ which is vertex-disjoint from the perimeter of~$H_3$.
    
    It follows that any graph~$G$ that contains $\mathscr{U}_{2(h_{\ref{prop:cleanwall}}(k)+2)+2}$ as a minor also contains~$H_4$ as a minor.
    Moreover, since~$H_4$ has maximum degree three, $G$ contains a subgraph~$B_1$ which is isomorphic to a subdivision of~$H_4$.
    Let us denote by~$P_1$ and~$P_2$ the two paths of~$B_1$ which correspond to the paths~$P_1'$ and~$P_2'$ of~$H_4$ respectively.
    Moreover, let~$B_1'$ denote the subgraph of~$B_1$ that corresponds to the subdivision of~$H_4'$.
    
    Then~$B_1'$ is an ${h_{\ref{prop:cleanwall}}(k)}$-wall which is vertex-disjoint from the perimeter of the ${(h_{\ref{prop:cleanwall}}(k)+2)}$-wall in~$B_1$.
    By \cref{prop:cleanwall} we may now find a clean $k$-wall~$B_2$ as a subwall of~$B_1'$.
    Let~$s_1$, $s_2$, $t_1$, $t_2$ be the four corners of~$B_2$ as encountered while traversing along its perimeter. 
    We may now enhance the two paths~$P_1$ and~$P_2$ to obtain two vertex-disjoint $B_2$-ears~$Q_1$ and~$Q_2$ such that~$Q_i$ has endpoints~$s_i$ and~$t_i$ for each~${i \in [2]}$. 
    
    Observe that the union of~$Q_1$, $Q_2$, and~$B_2$ is indeed a clean cross-wall of order~$k$ which is a subgraph of~$G$. 
\end{proof}

We now have everything in place to prove that every $2$-connected graph excluding the singly parity-crossing grids as \oddminors\ must have small $(\mathcal{B}\cup\mathcal{P})$-blind-treewidth. 

\begin{lemma}
    \label{lem:find-spc}
    There exist computable functions~$f_{\ref{lem:find-spc}}$,~$g_{\ref{lem:find-spc}}$ and an algorithm that, 
    given a positive integer~$k$ and a $\mathscr{U}_{f_{\ref{lem:find-spc}}(k)}$-expansion in a non-bipartite $2$-connected graph~$G$ as inputs, 
    in time~{$\mathcal{O}(g{\ref{lem:find-spc}}(k)\abs{V(G)}\log\abs{V(G)})$}, 
    finds an odd $\mathscr{C}^{i}_k$-expansion for some~${i \in [3]}$. 
\end{lemma}

\begin{proof}
    Let~$f_{\ref{lem:InsideOutWall}}$ and~$f_{\ref{lemma:cleancrosswall}}$ denote the function from \Cref{lem:InsideOutWall} and \Cref{lemma:cleancrosswall}, respectively. 
    We set~${f(k) \coloneqq f_{\ref{lemma:cleancrosswall}}(2(f_{\ref{lem:InsideOutWall}}(2k)+4)+1 ) }$.
    
    By \Cref{lemma:cleancrosswall}, $G$ contains a clean cross-wall~$\tilde{H}$ of order~${2 (f_{\ref{lem:InsideOutWall}}(2k) + 4 ) + 1}$ as a subgraph.
    Let~$W$ denote the clean~${(2 (f_{\ref{lem:InsideOutWall}}(2k) + 4 ) + 1)}$-wall in~$\tilde{H}$ and~$P_1$ and~$P_2$ be the two vertex-disjoint $W$-ears in~$\tilde{H}$ that form the cross on its four corners. 
    Since~$W$ is clean, the mutual distance of its corners within~$W$ is always even. 
    Moreover, note that~$\tilde{H}$ can be found in linear time (depending on~$k$). 

    \begin{claim}
        \label{claim:bipartitecrossgrid}
        Every non-bipartite clean cross-wall~$H$ of order at least~${f_{\ref{lem:InsideOutWall}}(2k)}$ contains~$\mathscr{C}_k^1$ or~$\mathscr{C}_k^2$ as an \oddminor. 
    \end{claim}

    \begin{proofofclaim}
        Observe that~$W$ is a bipartite graph, the only way how~$H$ could not be bipartite is if one of the~$P_i$ is of odd length. 
        By evoking \Cref{lem:InsideOutWall} we may find, as a subgraph of~$H$, a graph~$H'$ which is obtained from a clean $2k$-wall~$W'$ by adding two vertex-disjoint $W'$-ears~$P_1'$ and~$P_2'$ such that~$P_i'$ has endpoints~$s_i'$ and~$t_i'$ for both~${i \in [2]}$ and the vertices~$s_1'$, $s_2'$, $t_1'$, $t_2'$ are vertices of degree two of the central brick of~$W'$ all belonging to the same colour class for every proper $2$-colouring of~$W$ such that every $a$-$b$-path on the central brick of~$W'$ contains a branch vertex of~$W'$ for all distinct choices of~${a,b \in \{ s_1',s_2',t_1',t_2'\}}$. 
        Moreover, $P_1'$ and~$P_2'$ can be chosen such that~${P_i \subseteq P_i'}$ and~$P_i'$ is of odd length if and only if~$P_i$ is of odd length for both~${i \in [2]}$.
    
        Let~$\eta$ be a $\mathscr{G}_{2k}$-expansion in~$W'$. 
        Let~$x_1$, $x_2$, $y_1$, $y_2$ be the four vertices of the central four-cycle in~$\mathscr{G}_{2k}$ in the order as encountered when following along this cycle. 
        Without loss of generality we may assume~${s_i \in V(\eta(x_i))}$ and~${t_i \in V(\eta(y_i))}$ for both~${i \in [2]}$. 
        Notice that~$x_i$ and~$y_i$ belong of the same colour class of every proper $2$-colouring of~$\mathscr{G}_{2k}$ for both~${i \in [2]}$ while~$x_1$ and~$x_2$ belong to different colour classes. 
        Fix a proper $2$-colouring~$c_1$ of the inflated copy~${U_{\eta} \subseteq W}$ of~$\mathscr{G}_{2k}$ and a proper $2$-colouring~$c_2$ of~$\mathscr{G}_{2k}$ and obtain a new colouring~$c$ of~$U_{\eta}$ by setting~${c(v) \coloneqq 3-c_1(v)}$ for every vertex~${v \in V(U_{\eta})}$ such that there exists a vertex~${u \in V(\mathscr{G}_{2k})}$ with~${c_2(u) = 2}$ and~${v \in V(\eta(u))}$. 
        It follows that~$\eta$ is an odd~$\mathscr{G}_{2k}$-expansion. 
    
        Finally, notice that, if one of the~$P_i$ is of odd length and the other one is of even length, ${U_{\eta} \cup P_1' \cup P_2'}$ contains~$\mathscr{C}^2_k$ as an \oddminor.
        If both of the~$P_i$ are of odd length, then~${U_{\eta} \cup P_1' \cup P_2'}$ contains~$\mathscr{C}^1_k$ as an \oddminor.
        And if both of the~$P_i$ are even, $H$ is bipartite, as desired. 
    \end{proofofclaim}

    So we may assume that~$\tilde{H}$ is a $2$-connected and bipartite graph. 
    By \Cref{lemma:oddear}, since $G$ is $2$-connected and non-bipartite, there is an odd $\tilde{H}$-ear~$Q$ (which can be found in linear time). 
    We denote the endpoints of~$Q$ by~$x$ and~$y$. 

    What follows is a case distinction on where exactly the vertices~$x$ and~$y$ lie in~$\tilde{H}$. 
    Before we dive into the case distinction, let us fix some conventions. 
    As~$W$ is a clean ${(2 (f_{\ref{lem:InsideOutWall}}(2k) + 4 ) + 1)}$-wall there exist four clean ${(f_{\ref{lem:InsideOutWall}}(2k) + 4)}$-subwalls~$W_i'$, ${i \in [4]}$ of~$W$ such that for each pair of distinct~${i,j \in [4]}$ it holds that~${W'_i \cap W_j' \subseteq C_i' \cap C_j'}$ where~$C_i'$ and~$C_j'$ denote the perimeters of~$W_i'$ and~$W_j'$ respectively. 
    It follows that~${W = \bigcup_{i \in [4]} W_i'}$. 
    For each~${i \in [4]}$, let~$W_i$ be the unique ${f_{\ref{lem:InsideOutWall}}(2k)}$-subwall of~$W_i'$ such that every brick of~$W_i$ is also a brick of~$W_i'$ and the central brick of~$W_i$ is the central brick of~$W_i'$. 
    Note that the~$W_i$ are pairwise vertex-disjoint and every path from a vertex of~$W_i$ to the perimeter of~$W_i'$ contains at least three branch vertices of~$W_i'$ which are not branch vertices of~$W_i$. 
    See \Cref{figure:wallquad} for an illustration. 

    Now clearly~${\{x,y\} \subseteq {\bigcup_{j \in [4] \setminus \{ i \}} V(W_j')} \cup V(P_1) \cup V(P_2)}$ for some~${i \in [4]}$. 
    Observe that there exist vertex-disjoint $W_i$-ears~$\hat{P}_1$ and~$\hat{P}_2$ in~$H$ such that ${\hat{H} \coloneqq W_i \cup \hat{P}_1 \cup \hat{P}_2}$ is a clean cross-wall of order~${f_{\ref{lem:InsideOutWall}}(2k)}$ and~${\{x,y\} \subseteq V(\hat{P}_1) \cup V(\hat{P}_2)}$. 
    We distinguish two cases. 
    
    \begin{figure}[ht]
        \begin{center}
            \includegraphics[scale=0.6]{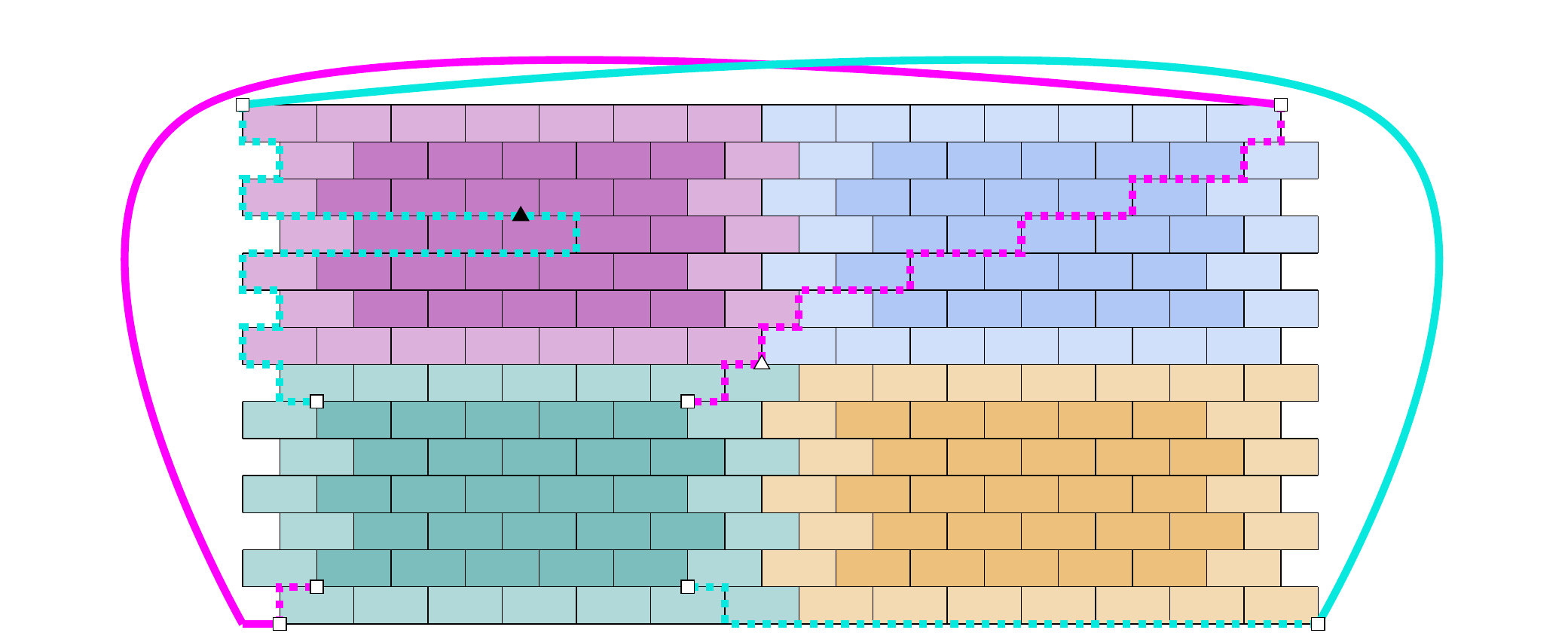}
        \end{center}
        \caption{Splitting up a wall into its four quadrants. Using the dashed paths, we can extend the crossing ears of~$W$ to crossing ears of one subwalls~$W_i$ while collecting the vertices~$x$ and~$y$, marked by triangles. }
        \label{figure:wallquad}
    \end{figure}

    \medskip
    
    \textbf{Case 1:} There exists~${j \in [2]}$ such that both~$x$ and~$y$ belong to $V(\hat{P}_j)$.
    \smallskip
    
    We may now replace~$\hat{P}_j$ by an odd $W_i$-ear~${Q' \subseteq \hat{P}_j \cup Q}$. 
    It follows that ${W_i \cup Q' \cup \hat{P}_{3-j}}$ is a non-bipartite clean cross-wall of order~${f_{\ref{lem:InsideOutWall}}(2k)}$. 
    Hence, this case concludes with the Claim. 
    \medskip

    \textbf{Case 2:} ${x \in V(\hat{P}_1)}$ and~${y \in V(\hat{P}_2)}$.
    \smallskip

    As in the proof of the \Cref{claim:bipartitecrossgrid}, 
    we can apply \Cref{lem:InsideOutWall} to find an odd $\mathscr{U}_k$-expansion in~${W_i \cup \hat{P}_1 \cup \hat{P}_2}$. 
    We will slightly modify this to an odd $U'$-expansion $\eta'$ of the graph~$U'$ obtained from the singly-crossing grid of order~$k$ by subdividing each of the two crossing edges exactly once. 
    Notice that any such $U'$-expansion $\eta'$ must assign both~$x$ and~$y$ to some branch sets and, in particular, these branch sets must be distinct. 
    Fix a $2$-colouring~$c_1$ of~$H$ and a $2$-colouring~$c_2$ of~$U'$.
    We choose~$\eta'$ in a way such that for the vertices~${w_x,w_y}$ with~${x \in V(\eta'(w_x))}$ and~${y \in V(\eta'(w_y))}$ it holds that~${c_2(w_x) = c_2(w_y)}$. 
    See \Cref{figure:wallc3} for an illustration. 

    \begin{figure}[ht]
        \begin{center}
            \includegraphics[scale=0.3]{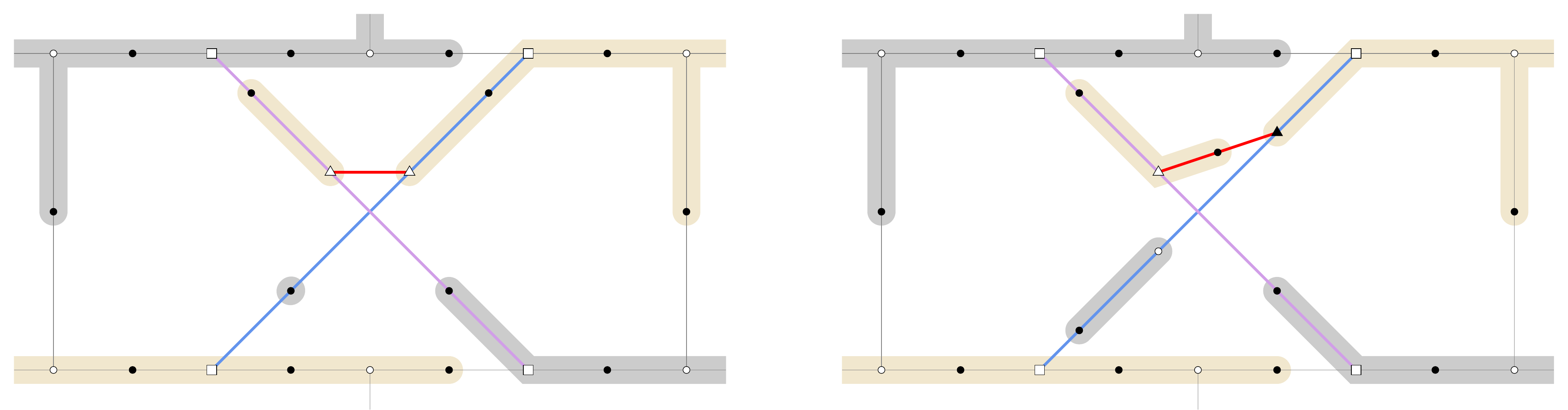}
        \end{center}
        \caption{The two ways how to extend the odd $\mathscr{G}_{2k}$-expansion in the central brick of a wall to an odd $\mathscr{C}_k^3$-expansion.  Inverting the colourings of the grey branch sets yields that this expansion is odd. }
        \label{figure:wallc3}
    \end{figure}

    If~$x$ and~$y$ belong to the same colour class under~$c_1$, as~$Q$ is parity breaking it must be of odd length and thus we may further augment~$\eta'$ to obtain an odd $\mathscr{C}^3_k$-expansion as depicted in the left of \Cref{figure:wallc3}. 

    So assume that~$x$ and~$y$ belong to different colour classes under~$c_1$.
    Moreover, in this case~$Q$ must be of even length and thus, we may further augment~$\eta'$ to obtain an odd $\mathscr{C}^3_k$-expansion. 
    This is illustrated on the right of \Cref{figure:wallc3}. 
\end{proof}

With this we are ready for the proof of \Cref{thm:singleoddcrossstructure}. 
As before, we prove a slightly stronger version.

\begin{theorem}\label{thm:singleoddcrossstructure-actual}
    It holds that $({\mathcal{B} \cup \mathcal{P})\text{-}\mathsf{blind}\text{-}\mathsf{tw} \sim \mathsf{big}_{\mathfrak{C}}}$. 
    In particular, there exist computable functions ${f_{\ref{thm:singleoddcrossstructure-actual}},g_{\ref{thm:singleoddcrossstructure-actual}}\colon\mathbb{N}\to\mathbb{N}}$ 
    and an algorithm that, given a positive integer~$k$ and a graph~$G$ as inputs, 
    in time $\mathcal{O}(g_{\ref{thm:singleoddcrossstructure-actual}}(k) \abs{V(G)}^3 )$, 
    finds either
    \begin{enumerate}
        \item an odd expansion of a singly parity-crossing grid of order~$k$, or
        \item a tree-decomposition~${(T,\beta)}$ for~$G$ of $(\mathcal{B} \cup \mathcal{P})$-blind-width~$f_{\ref{thm:singleoddcrossstructure-actual}}(k)$, where, in particular, for every node~${t \in V(T)}$ with~${\abs{\beta(t)} > f_{\ref{thm:singleoddcrossstructure-actual}}(k)}$, we have that either 
            \begin{itemize}
                \item $G[\beta(t)]$ a bipartite block of~$G$ and the adhesion of~$t$ is at most one, or
                \item the torso on~$t$ is a planar quasi-$4$-component of~$G$ and the adhesion of~$t$ is at most~$3$. 
            \end{itemize}
    \end{enumerate}
\end{theorem}

\begin{proof}
    Let~$f_{\ref{lem:find-spc}}$ and~$f_{\ref{prop:singlecrossing}}$ be the functions from \Cref{lem:find-spc} and \Cref{prop:singlecrossing}, respectively, and let~${f_{\ref{thm:singleoddcrossstructure-actual}}(k) \coloneqq f_{\ref{prop:singlecrossing}}(f_{\ref{lem:find-spc}}(k))}$. 
    Recall Observation~\ref{obs:blind-tw-blocks}. 
    We apply \Cref{prop:cutvertextree}. 
    To each non-bipartite block, we apply \Cref{prop:singlecrossing} to either find an~$\mathscr{U}_{f_{\ref{lem:find-spc}}(k)}$-expansion in a non-bipartite $2$-connected graph, or a tree-decomposition with the properties as in \Cref{prop:singlecrossing}. 
    If for some block we are in the first case, we get an odd $\mathscr{C}^{i}_k$-expansion from \Cref{lem:find-spc}. 
    Otherwise, we combine the tree-decomposition from \Cref{prop:cutvertextree} with the tree-decompositions given by \Cref{prop:singlecrossing}. 
    Now every bag of size larger than~$f_{\ref{thm:singleoddcrossstructure-actual}}(k)$ is either a bipartite block of~$G$ (and hence~$(G,\beta(t)) \in \mathcal{B}$), or its torso is a planar quasi-$4$-component (and hence~$(G,\beta(t)) \in \mathcal{P}$). 
    Thus the $(\mathcal{B}\cup\mathcal{P})$-blind-width of this tree-decomposition is at most~$f_{\ref{thm:singleoddcrossstructure-actual}}(k)$, as desired. 

    With \Cref{cor:bpb-f(BP-tw)}, the asymptotic equivalence of the parameters follows. 
\end{proof}

\section{%
\texorpdfstring{\textsc{Maximum Independent Set} on graphs of bounded $\mathcal{B}$-blind-treewidth}{Maximum Independent Set on graphs of bounded B-blind-treewidth}}
\label{sec:maxindependentset}

We solve the more general \textsc{Maximum Weight Independent Set} problem on vertex-weighted graphs $(G,w)$ where~$G$ is of bounded $\mathcal{B}$-blind-treewidth. 
A \emph{vertex-weighted graph}~${(G,w)}$ consists of a graph~$G$ and a function~${w \colon V(G) \to \mathbb{N}}$. 
If~${H \subseteq G}$ we denote by~${(H,w)}$ the \emph{vertex-weighted subgraph} of~${(G,w)}$. 
Then, using the constant-1 weighting~$\mathbf{1}_{V(G)}$ implies~\Cref{thm:maxindependentset}. 
The only change for the weighted version that might occur is that the coding length of the integers in the image of~$w$ has to be taken into account for the running time. 

In the following, we will use that fact that instead of \textsc{Maximum Weight Independent Set} one can solve \textsc{Minimum Weight Vertex Cover} which famously reduces to network flows on bipartite graphs. 
By using Dinic's algorithm~\cite{dinic1970} we obtain a running time of~$\mathcal{O}(n^3)$ for vertex-weighted bipartite $n$-vertex graphs where the total weight is in~$\mathcal{O}(n)$. 

\begin{proposition}[follows from~\cite{dinic1970}]\label{prop:bipartiteindependentset}
    There exists an algorithm that, given a vertex-weighted bipartite graph~$(G,w)$ with $\sum_{v\in V(G)}w(v)\in \mathcal{O}(\abs{V(G)})$ as an input, 
    in time $\mathcal{O}(\abs{V(G)}^3)$, 
    computes a maximum weight independent set of~$(G,w)$. 
\end{proposition}

While \Cref{prop:bipartiteindependentset} allows us to deal with bipartite blocks, to deal with non-bipartite blocks, we also need the following. 

\begin{proposition}[see for example~\cite{NiedermeierBook}]
    \label{prop:twindependentset}
    There exists an algorithm that, given a positive integer~$k$ and a vertex-weighted graph~${(G,w)}$ with~${\sum_{v \in V(G)} w(v) \in \mathcal{O}(\abs{V(G)})}$ of treewidth at most~$k$ as inputs, 
    in time~${2^{\mathcal{O}(k)}\abs{V(G)}}$, 
    computes a maximum weight independent set of~$(G,w)$. 
\end{proposition}

Given a rooted tree-decomposition~${\mathcal{T} = (T,r,\beta)}$ we denote by~$\mathcal{T}_t$ the rooted tree-decomposition ${(T_t,t,\beta_t)}$ where~${t \in V(T)}$, $T_t$ is the subtree of~$T$ rooted at~$t$, and~$\beta_r$ is the restriction of~$\beta$ to~$T_t$. 
Moreover, we denote by~${(G_{T_t},w)}$ the vertex-weighted subgraph of~${(G,w)}$ induced by the vertex set~${\bigcup_{d \in V(T_t)} \beta(d)}$ and by~${(G_t,w)}$ the vertex-weighted subgraph of~${(G,w)}$ induced by~${\beta(t)}$. 

Without loss of generality, let us assume that~$G$ is connected. 
Let~${(T,\beta)}$ be the tree-decomposition from \Cref{prop:cutvertextree} displaying the blocks of~$G$. 
As stated in \Cref{prop:cutvertextree}, ${(T,\beta)}$ can be found in linear time.
We pick an arbitrary root~${r \in V(T)}$ and treat~${(T,\beta)}$ as a rooted tree-decomposition~${\mathcal{T} = (T,r,\beta)}$. 
Let~${t \in V(T) \setminus \{ r \}}$ be any non-root vertex and let~${d \in V(T)}$ be the unique neighbour of~$t$ which is closer to~$r$ in~$T$ than~$t$. 
Then there exists a unique vertex~$v_t$ such that~${\{ v_t \} = \beta(t) \cap \beta(d)}$.
The table for our dynamic programming on~$\mathcal{T}$ will consist of exactly two numbers for every~${t \in V(T) \setminus \{ r \}}$, namely
\begin{itemize}
    \item $\mathbf{in}_t$, the maximum weight of an independent set of $(G_{T_t} - N_{G_{T_t}}[v_t],w)$\footnote{$N_G[v]$ denotes the \emph{closed neighbourhood} of~$v$ in~$G$, that is $N_G[v]=\{ v\}\cup N_G(v)$.} and 
    \item $\mathbf{out}_t$, the maximum weight of an independent set of $(G_{T_t}-v_t,w)$.
\end{itemize}
Notice that~$\mathbf{in}_t+w(v_t)$ is the maximum weight of an independent set of~${(G_{T_t},w)}$ that contains~$v_t$ while~$\mathbf{out}_t$ is the maximum weight of an independent set of~${(G_{T_t},w)}$ that avoids~$v_t$. 

Finally, we need to describe an auxiliary graph that will allow us to merge the tables. 
Assume we are given some~${t \in V(T)}$ let~${d_1, \dots, d_{\ell}}$ be its successors in~$T$ and suppose that we have already computed the values~$\mathbf{in}_{d_i}$ and~$\mathbf{out}_{d_i}$ for all~${i \in [\ell]}$. 
Let~${(G_t^+,w_t^+)}$ be the vertex-weighted graph obtained from~$G_t$ as follows.
For each~${i \in [\ell]}$ introduce a new vertex~$x_i$ adjacent to the vertex~$v_{d_i}$ in~$G_t$.
We set~${w_t^+(x_i) \coloneqq \mathbf{out}_{d_i}}$. 
Let~${A \subseteq \beta(t)}$ be the set of all vertices~$a$ such that there exists~${i \in [\ell]}$ where~${a = v_{d_i}}$. 
Given some~${a \in A}$ let~${I_a \subseteq [\ell]}$ be the set of all~${i \in [\ell]}$ such that~${a = v_{d_i}}$. 
We set~${w_t^+(a) \coloneqq w(a) + \sum_{i \in I_a} \mathbf{in}_{d_i}}$.
For all vertices~$v$ in~${\beta(t) \setminus A}$ we set~${w_t^+(v) \coloneqq w(v)}$.
Notice that, if~${\sum_{v \in V(G)} w(v) \in \mathcal{O}(\abs{V(G)})}$, then also~${\sum_{v \in V(G_t^+)} w_t^+(v) \in \mathcal{O}(\abs{V(G)})}$.

Now, everything is in place for the proof of \cref{thm:maxindependentset}. 

\begin{proof}[Proof of \Cref{thm:maxindependentset}]
    Let us first observe that, to find the decomposition as claimed in the assertion, we may first call the algorithm from \Cref{prop:cutvertextree} to obtain, in time~${\mathcal{O}(\abs{V(G)}+\abs{E(G)})}$, a rooted tree-decomposition~${(T,r,\beta)}$ for~$G$ that displays its blocks by selecting some arbitrary root. 
    
    By \Cref{lem:find-spb} and \Cref{cor:bpb-f(B-tw)}, every non-bipartite block has treewidth at most~$f(k)$, where~$f$ denotes the function from \Cref{lem:find-spb}. 
    
    Let~$t$ be a leaf of~$T$.
    In this case there is no difference between~${(G_t,w)}$ and~${(G_t^+,w_t^+)}$ and thus we may use either \Cref{prop:bipartiteindependentset} or \Cref{prop:twindependentset}, depending on whether~$G_t$ is bipartite or of bounded treewidth, to compute the values~$\mathbf{in}_t$ and~$\mathbf{out}_{t}$.
    
    So we may assume~$t$ to not be a leaf and assume that the values~$\mathbf{in}_d$ and~$\mathbf{out}_d$ are given for all~${d \in V(T_t) \setminus \{ t \}}$. 
    In case~${t = r}$ let~$v_r$ be an arbitrary vertex of~$\beta(r)$, otherwise~$v_t$ is the vertex as defined above. 
    
    Now consider the vertex-weighted graph~${(G^+_t,w_t^+)}$ and observe that, if~$G_t$ is bipartite, then so is~$G_t^+$.
    Moreover, the treewidth of~$G_t^+$ equals the treewidth of~$G_t$. 
    Let~${d_1,\dots, d_{\ell}}$ denote the successors of~$t$ in~$T$ and let~${I_t \subseteq [\ell]}$ be the set of all~${i \in [\ell]}$ such that~${v_t = v_{d_i}}$. 
    To compute the value~$\mathbf{out}_t$ simply compute a maximum weight independent set of~${G_t^+ - v_t}$ by using \Cref{prop:bipartiteindependentset} or \Cref{prop:twindependentset}. 
    To compute the value~$\mathbf{in}_t$ we first compute the weight~$z$ of a maximum weighted independent set of~${G_t^+ - N_{G_t^+}[v_t]}$ as before and then set~${\mathbf{in}_t \coloneqq z + \sum_{i \in I_t} \mathbf{in}_{d_i}}$.
    
    We now have to check that the values~$\mathbf{in}_t$ and~$\mathbf{out}_t$ have been computed correctly.
    Let~$J_{\mathsf{in}}^+$ be a maximum weight independent set of~${(G^+_t - N_{G^+_t}[v_t], w^+)}$ and~${J^+_{\mathsf{out}}}$ be a maximum weight independent set of~${(G^+_t - v_t, w^+)}$. 
    Moreover, let~$J_{\mathsf{in}}$ be a maximum weight independent set of~${(G_{T_t} - N_{G}[v_t], w)}$ and~$J_{\mathsf{out}}$ be a maximum weight independent set of~${(G_{T_t} - N_{G}[v_t], w)}$. 
    We will only discuss the case of~$J_{\mathsf{in}}$ as the case of~$J_{\mathsf{out}}$ can be handled analogously. 
    
    Let~$W$ be the set obtained from~$G$ by taking the union of 
    \vspace{-6pt}
    \begin{itemize}
        \item the set~$J_{\mathsf{in}}^+\cap V(G)$, \item for each~${i \in [\ell] \setminus I_t}$, if~${v_{d_i} \in J_{\mathsf{in}}^+ \cap V(G)}$ then a maximum weight independent set~$A_i$ of~${(G_{T_{d_i}} -N_{G}[v_{d_i}], w)}$ and if~${x_i \in J_{\mathsf{in}}^+ \cap V(G)}$ then a maximum weight independent set~$A_i$ of~${(G_{T_{d_i}} - v_{d_i}, w)}$, and
        \item for every~${i \in I_t}$, a maximum weight independent set~$A_i$ of~${(G_{T_{d_i}} - N_{G}[v_t], w)}$. 
    \end{itemize}
    \vspace{-6pt}
    It follows that~$W$ is an independent set in~${G_{T_t} - N_{G}[v_t]}$.
    Hence, ${w(W) \leq w(J_{\mathsf{in}})}$.
    Moreover, ${w(W) = w^+(J_{\mathsf{in}})}$ and thus, we obtain~${w^+(J_{\mathsf{in}}) \leq w(J_{\mathsf{in}})}$. 
    
    Note that~$J_{\mathsf{in}}$ can also be partitioned into independent sets~$B_i$ of~${G_{T_{d_i}} - \beta(t)}$ for each~${i \in [\ell]}$ and an independent set~$S$ of~$G_t$. 
    We construct a set~$W^+$ as the union of the set~$S$, and, for each~${i \in [\ell]}$ where~${v_{d_i} \notin B_i}$, the set~${\{ x_i \}}$. 
    Observe that~${v_{d_i} \in S}$ for all other~${i \in [\ell]}$.
    Moreover, ${W^+ \setminus N_{G_t^+}[v_t]}$ is an independent set of~$G_t^+$.
    Since~${w(J_{\mathsf{in}}) = w^+(W^+)}$, it follows from the choice of~$J_{\mathsf{in}}^+$ that
    \begin{align*}
        w^+(W^+) \leq w^+(J_{\mathsf{in}}^+) + \sum_{i \in I_t} w^+(x_i). 
    \end{align*}
    So finally, we obtain~${w(J_{\mathsf{in}}) = \mathbf{in}_t}$. 
\end{proof}

\section{
\texorpdfstring{\textsc{Maximum Cut} on graphs of bounded $(\mathcal{B}\cup\mathcal{P})$-blind-treewidth}{Maximum Cut on graphs of bounded BuP-blind-treewidth}}
\label{sec:maxcut}

Similar to our approach for \textsc{Maximum Independent Set} we prove \Cref{thm:maxcut} by proving a more general form for edge-weighted graphs. 
An \emph{edge-weighted graph}~$(G,w)$ consists of a graph and a function~${w \colon E(G) \to \mathbb{N}}$. 
If~${H \subseteq G}$ we denote by~$(H,w)$ the \emph{edge-weighted subgraph} if~$(G,w)$. 
For a set of edges~$w(F)$, we denote by~$w(F)$ the total weight~${\sum_{e \in F} w(e)}$. 

Given~${X \subseteq V(G)}$ we denote the set~${\{ e \in E(G) \mid \abs{e \cap X} = 1 \}}$ by~${\partial_G(X)}$. 
A set~${F \subseteq E(G)}$ is a \emph{cut} of~$G$ if there exists a non-empty proper subset~${X \subseteq E(G)}$ where~${F = \partial_G(X)}$. 
A \emph{maximum weight cut} of an edge-weighted graph~${(G,w)}$ is a cut of maximum total weight.

Let~$p$ be a polynomial. 
An edge-weighted graph~$(G,w)$ is said to be \emph{$p$-bounded} if~${\max\{ w(e) \mid e\in E(G)\} \leq p(\abs{V(G)})}$. 

\begin{theorem}
    \label{thm:maxweightcut}
    There exists a computable function~$g_{\ref{thm:maxweightcut}}$ and an algorithm that, 
    given a positive integer~$k$, a polynomial~$p$, a $p$-bounded edge-weighted graph~${(G,w)}$, and a tree-decomposition of~${(\mathcal{B} \cup \mathcal{P})}$-blind-width~$k$ for~$G$ as inputs, 
    in time ${\mathcal{O}(g_{\ref{thm:maxweightcut}}(k) \abs{V(G)}^{\nicefrac{9}{2}} p(\abs{V(G)}))}$, 
    finds a maximum weight cut of a $p$-bounded edge-weighted graph~${(G,w)}$. 
\end{theorem}

\Cref{thm:maxcut} follows by applying this to~${(G,\mathbf{1}_{E(G)})}$ where~$p$ is constant. 

As before, we require subroutines for bipartite graphs and graphs that exclude a single-crossing grid. 

\begin{observation}
    \label{obs:bipartitecut}
    The edge set of a bipartite graph is a cut. 
\end{observation}

\begin{proposition}[Follows from~\cite{GalluccioLV2001,Kaminski2012,RobertsonS1993}]
    \label{prop:maxcutsinglecrossing}
    There is a function~$g_{\ref{prop:maxcutsinglecrossing}}$ and an algorithm that, 
    given a positive integer~$k$, a polynomial~$p$, and a $p$-bounded edge-weighted graph ${(G,w)}$ that excludes $\mathscr{U}_k$ as a minor, 
    in time~${\mathcal{O}(g_{\ref{prop:maxcutsinglecrossing}}(k) \abs{V(G)}^{\nicefrac{9}{2}} p(\abs{V(G)}))}$, 
    finds a maximum weight cut of~${(G,w)}$. 
\end{proposition}

Let us now discuss how cuts of a graph relate to the cuts of their blocks.

\begin{lemma}
    \label{lem:cutInBlocks}
    Let~$G$, let~${F \subseteq E(G)}$, and let~$\mathcal{B}$ denote the set of blocks of~$G$. 
    Then~${F}$ is a cut of~$G$ if and only if~${F \cap E(B)}$ is a cut of~$B$ for every block~$B$ of~$G$. 
    In particular, if~$w$ is an edge-weighting of~$G$, then~${w(F) = \sum_{B \in \mathcal{B}} w(F_B)}$. 
\end{lemma}

\begin{proof}
    An edge set~$F$ is a cut of a graph~$G$, if any only if~$F$ meets every cycle of~$G$ in an even number of edges (this follows directly from {Theorem~1.9.4} of~\cite{Diestel:GraphTheory5}). 
    Now since every edge of~$G$ is contained in a unique block and every cycle of~$G$ is contained in a unique block, the result follows from this characterisation of cuts. 
\end{proof}

With this, everything is in place for the proof of \Cref{thm:maxweightcut}. 

\begin{proof}[Proof of \cref{thm:maxweightcut}]
    By \Cref{lem:cutInBlocks} it suffices to find a maximum weight cut~$F_t$ in every block $G_t$ of $G$, where $t\in V(T)$, to find a maximum weight cut in~$G$.
    
    By \Cref{lem:find-spc} and \Cref{cor:bpb-f(BP-tw)}, every block is either bipartite or excludes~${\mathscr{U}_{f(k)}}$ as a minor, where~$f$ denotes the function from \Cref{lem:find-spc}. 
    We may now use \Cref{prop:maxcutsinglecrossing} and Observation~\ref{obs:bipartitecut} to obtain a maximum weight cut for each blocks of~$G$. 
    By \Cref{lem:cutInBlocks}, their union is a maximum weight cut of~$G$. 
\end{proof}

\section{Conclusion}

Recall that a graph~$H$ has the \emph{Erd\H{o}s-P\'osa property} for minors if there exists a function~$f$ such that for every~$k$, every graph~$G$ either contains~$k$ pairwise vertex-disjoint inflated copies of~$H$, or there is a set~${S \subseteq V(G)}$ of size at most~${f(k)}$ such that~${G-S}$ is $H$-minor-free.

A consequence of the Grid Theorem of Robertson and Seymour is that~$H$ has the Erd\H{o}s-P\'osa property for minors if and only if~$H$ is a planar graph.

By exchanging minor for \oddminor, we may define the Erd\H{o}s-P\'osa property for \oddminors.
By using \cref{prop:cleanwall} and arguments from~\cite{Reed1999} it is possible to prove that~$H$ has the Erd\H{o}s-P\'osa property for \oddminors\ if and only if~$H$ is planar and bipartite.

A graph~$H$ is said to have the \emph{half-integral Erd\H{o}s-P\'osa property} if there is a function~$f$ such that for every~$k$, every graph~$G$ either contains a collection of~$k$ inflated copies of~$H$ such that no vertex belongs to more than two of these inflated copies, or there exists a set~${S \subseteq V(G)}$ of size at most~${f(k)}$ such that~${G-S}$ is $H$-minor-free.

Liu recently proved that every graph has the half-integral Erd\H{o}s-P\'osa property for minors~\cite{Liu2022}. 
It appears natural to ask if this result can be extended to \oddminors.

Regarding our algorithmic results, it is well known that \textsc{Maximum Independent Set} is NP-complete on planar graphs~\cites{GareyJ1977,Murphy1992} and thus one cannot hope to find an FPT-algorithm with parameter ${(\mathcal{B}\cup\mathcal{P})}$-blind-treewidth.
\textsc{Maximum Cut}, on the other hand, is tractable on graphs excluding~$K_5$ as an \oddminor.
Finding the \oddminor-closed graph classes beyond classes of bounded $\mathcal{B}$-blind treewidth, or ${(\mathcal{B}\cup\mathcal{P})}$-blind-treewidth respectively, where one of these problems remains tractable appears to be worthwhile.

We conclude this paper by formulating the three questions above explicitly as \textbf{leading questions} for further directions of research.

\begin{question}
    Does every (connected) graph~$H$ have the half-integral Erd\H{o}s-P\'osa property for \oddminors?
\end{question}

\begin{question}
    For which (finite) families~$\mathcal{H}$ of graphs is there a polynomial time algorithm for \textsc{Maximum Independent Set} on the class of graphs that exclude all members of~$\mathcal{H}$ as \oddminors?
\end{question}

\begin{question}
    For which (finite) families~$\mathcal{H}$ of graphs is there a polynomial time algorithm for \textsc{Maximum Cut} on the class of graphs that exclude all members of~$\mathcal{H}$ as \oddminors?
\end{question}

\printbibliography

\end{document}